\documentclass[11pt]{amsart}
\usepackage{preamble}

\usepackage{tikz-cd}
\newcommand{\rank}{\operatorname{rank}}

\newcommand{\LC}{\operatorname{LC}}
\newcommand{\Val}{\operatorname{Val}}

\newcommand{\Span}{\operatorname{Span}}

\newcommand{\id}{\operatorname{id}}

\newcommand{\Pic}{\operatorname{Pic}}

\newcommand{\Aut}{\operatorname{Aut}}

\newcommand{\Star}{\operatorname{Star}}

\newcommand{\Link}{\operatorname{Link}}
\newcommand{\Gal}{\operatorname{Gal}}

\newcommand{\coreg}{\operatorname{coreg}}
\newcommand{\Ex}{\operatorname{Ex}}

\newcommand{\Cl}{\operatorname{Cl}}

\newcommand{\depth}{\operatorname{depth}}

\newcommand{\Diff}{\operatorname{Diff}}

\newcommand{\nklt}{\operatorname{nklt}}

\newcommand{\DMR}{\mathcal{DMR}}

\begin{document}
\title[On quasi-Albanese morphisms for log canonical Calabi-Yau pairs II]{On quasi-Albanese morphisms for log canonical Calabi-Yau pairs II}


\begin{abstract}
Centering on quasi-Albanese morphisms, we prove several structural theorems for Calabi-Yau pairs.
\end{abstract}

\author{Yiming Zhu}
\address{School of Mathematical Science, University of Science and Technology of China, Hefei 230026, P.R.China.} \email{zhuyiming@ustc.edu.cn}

\maketitle

\setcounter{tocdepth}{1}
\tableofcontents

\section{Introduction}

Let $X$ be a smooth Calabi-Yau variety. The Bogomolov-Beauville-Yau decomposition states that a finite \'etale cover of $X$ is a product of an abelian variety, a strict Calabi-Yau variety, and an irreducible symplectic variety \cite{beauville1983some}. In a series of works \cite{GKP16,Dru18,GGK19,horing2019algebraic,CampanaCaoMatsumura21,patakfalvi2019beauville,matsumura2021structure}, such a decomposition has been generalized to projective klt pairs. 

In this note, we consider log canonical Calabi-Yau pairs $(X,D_X)$. This means that $(X,D_X)$ has log canonical singularities and $K_X+D_X\sim_{\mathbb Q}0$. Bernasconi, Filipazzi, Patakfalvi, and Tsakanikas proved that the Albanese morphism need not be isotrivial \cite{bernasconi2024strong}. On the other hand, let $Z$ be a minimal lc center of $(X,D_X)$.  Rösler proved that if there is a torus $T\subset\Aut(X)$ of dimension $\codim_X Z$ preserving $D_X$, then the Albanese morphism is generically birationally
isotrivial \cite{rosler2026birational}. 

This note continues the study of quasi-Albanese morphisms \cite{iitaka1976logarithmic,fujino2024quasi} for Calabi-Yau pairs in \cite{zhu2025quasi}. Recall that a quasi-abelian variety $G$ is an extension of an abelian variety $A$ by an algebraic torus $T$. Such a variety is also called a semi-abelian variety. We call a compactification of $G$ toric if it is obtained from a toric compactification of the torus fibers of $G\to A$. See Definition \ref{def:Qabelian} for the precise construction.

Our first theorem describes Calabi-Yau pairs whose quasi-Albanese morphisms are proper in codimension one.

\begin{thm}\label{thm1}
Let $(X,D_X)$ be a projective log canonical Calabi-Yau pair, let $X^0$ be its klt locus, and let $\alpha:X^0\to G$ be the quasi-Albanese morphism. If $\alpha$ is proper in codimension one, then the following hold.
 \begin{enumerate}
     \item $\alpha$ is proper.
     \item There is an open subset $U\subset X$ such that all lc centers intersect $U$, $D_X|_{U}=D_X^{=1}|_U$, and $(U,D_X|_{U})$ is toroidal without self-intersection. 
     \item There is a projective toric compactification $(Y,D_Y)$ of $G$, and a morphism $f:X\to Y$ extending $\alpha$ such that (4) holds.
     \item There is a finite Galois cover $q:Z\to Y$ which is \'etale over $G$ such that there is an isomorphism $(X,D_X)\times_YZ\cong (F,D_F)\times (Z,D_Z)$ over $Z$. Here $D_Z=q^{-1}D_Y$ and $X\times_YZ$ means the normalization of the usual fiber product.
 \end{enumerate}    
\end{thm}

A morphism $f:X\to Y$ between normal varieties is called proper in codimension one if there is a codimension $\geq2$ closed subset $Z\subset Y$ such that $f$ is proper over $Y\setminus Z$.

By Lemma \ref{lem:toric}, the pair $(Z,D_Z)$ in the theorem is also a toric compactification of a quasi-abelian variety.

The Calabi-Yau condition is important. Example \ref{exa1} shows that, under the weaker condition $\kappa(K_X+D_X)=0$, properness in codimension one does not imply properness. Example \ref{exa2} shows that without properness in codimension one, such a compactification may not exist.

Theorem \ref{thm1} can also be understood in terms of coregularity and complexity. See Remark \ref{rem}.

The following corollary generalizes \cite[Corollary 1.5]{bergner2021lipman} and \cite[Theorem 1.6]{zhu2025quasi}.

\begin{cor}\label{cor}
Let $(X,D_X)$ be a projective log canonical Calabi-Yau pair, let $X^0$ be its klt locus, and let $\alpha:X^0\to G$ be the quasi-Albanese morphism. Assume that $\dim X=\dim G$. Then $(X,D_X)$ is a toric compactification of $G$. In particular, $D_X$ is reduced and $X\setminus D_X\cong G$. 
\end{cor}

Without the properness condition, we prove a decomposition for dual complexes up to homotopy equivalence.

\begin{thm}\label{thm2}
Let $(X,D_X)$ be a projective log canonical Calabi-Yau pair, let $X^0$ be its klt locus, and let
$\alpha:X^0 \to G$ be the quasi-Albanese morphism. Let $(\Pp_A,H)$ be a standard compactification of $G$ (see Definition \ref{def:Qabelian}). Then the following holds.
\begin{enumerate}
    \item 
There is a crepant birational contraction $\pi:(W,D_W=\pi^{-1}_*D_X+\Ex(\pi))\dashrightarrow (X,D_X)$ from a $\Qq$-factorial dlt Calabi-Yau pair $(W,D_W)$ such that the rational map $f:(W,D_W)\to\Pp_A$ is a morphism.

\item For the morphism $f:(W,D_W)\to\Pp_A$ given in (1), there is a finite Galois morphism $q:(Z,D_Z)\to (\Pp_A,H)$ which is \'etale over $G$ such that the following hold. Let $(F,D_F)$ be the restriction of $(W,D_W)$ to the geometric generic fiber of $f$. 
Let $W_Z$ be the normalization of $W\times_{\Pp_A}Z$, let $p:W_Z\to W$ be the finite cover, and let $D_{W_Z}$ be divisor such that $K_{W_Z}+D_{W_Z
}=p^*(K_W+D_W)$. Then the dual complex $\Dd(D_{W_Z}^{=1})$ is homotopy equivalent to the join $\Dd(D_F^{=1})*\Dd(D_Z)$.    
\end{enumerate}
\end{thm}

Koll\'ar and Xu asked whether $\DMR(X,D_X)\cong\Ss^k/G$ for some finite group $G\subset O(k+1)$ \cite[Question 4]{KollarXu2016dual}. They gave a positive answer in dimension $\leq 4$. Theorem \ref{thm2} provides evidence for this question by relating the dual complex of a Calabi-Yau pair to those of its quasi-Albanese fiber and base. Mauri's work on dual complexes on Mori fibrations \cite{mauri2020dual} also motivates the formulation and the proof of Theorem \ref{thm2}.

\subsection*{Outline of the proofs}
We briefly explain the main ideas. We first prove
Theorem \ref{thm1} for dlt pairs and then extend it to lc
pairs. 

For a dlt pair, we use properness to show that the model
$f\colon (W,D_W)\to \Pp_A$ in Theorem \ref{thm2}(1) is klt
over the generic point of $\Pp_A$. Ambro's isotriviality theorem \cite{Ambro2005moduli} then shows that the monomial
valuation space $\operatorname{QM}(X,D_X)$ forms a fan. This fan gives a
toric compactification $(Y,D_Y)$ of $G$. A standard toric
argument shows that $\alpha$ extends to a morphism $X\to Y$.
The product decomposition then follows from isotriviality. See Section 2.

For lc pairs, we use the criterion of Moraga and Svaldi \cite{moraga2025geometric}
to show that $(X,D_X)$ is toroidal at the general points of its lc
centers. This allows us to define the monomial valuation space
$\Sigma(X,D_X)$ using the toroidal locus. The rest of the argument is
the same as in the dlt case. See Section 3.

We now outline the proof of Theorem \ref{thm2}.
Part (1) is proved in \cite{zhu2025quasi}, and we focus on part (2). We write $D^v$ for the vertical
part of a divisor $D$ with respect to the morphism under discussion.

First, the morphism
$\Sigma(f):\Sigma(W,D_W^v)\to\Sigma(\Pp_A,H)$ between cone complexes induces a
piecewise-linear map
$ \Dd(f):\Dd(D_W^v)\to\Dd(H)$ between complexes.
This map is compatible with adjunction to horizontal lc centers
(Lemma \ref{lem:slc}). If  $D_W$ is vertical, isotriviality
shows that $\Dd(f)$ is a homotopy equivalence
(Proposition \ref{prop:vertical}). For the general case, we prove the same statement by running a $(W,D_W^v)$-MMP and using induction on the
dimension. This follows the idea of \cite[Theorem 4]{de2017dual}. See
Proposition \ref{prop:inductivezero}.

Next, suppose that $f|_S$ has connected fibers for every horizontal
lc center $S$. The homotopy equivalence for the vertical complexes,
together with a simplicial-complex argument, gives the desired join
decomposition (Proposition \ref{prop:connected}). Finally, Koll\'{a}r's source-spring
theory \cite[Section 4.5]{Kollar2013sing} provides a finite base change for
which the restrictions to all horizontal lc centers have connected
fibers. This finishes the proof. See Proposition \ref{prop:basechange}.

\subsection*{Acknowledgments}
The author thanks Professors Zhan Li, Lei Zhang, Junchao Shentu, Joaqu\'in Moraga, and Osamu Fujino for helpful discussions and comments. The author completed part of the work during a visit to Professor Zhan Li and is grateful for his hospitality. Professor Mao Sheng supported this work through the CAS Project for Young Scientists in Basic Research (Grant No. YSBR-032). 

\subsection*{AI disclosure}
The author proposed the main ideas, definitions, proofs, and overall direction of this study. AI tools were used to refine technical details and locate relevant references.

\subsection*{Notation and conventions}
We work over the field of complex numbers $k=\Cc$.

For singularities of pairs and the minimal model program, we refer to \cite{KM98,Kollar2013sing}. We call a pair $(X,D_X)$ Calabi-Yau if $K_X+D_X\sim_{\Qq}0$. 

A dominant morphism with irreducible general fibers between normal varieties is called a fibration. Let $f:X\to Y$ be a fibration. Let $D$ be an effective $\Qq$-divisor on $X$. Then $ D$ decomposes as $D=D^h+D^v$ into effective $\Qq$-divisors such that $D^v$ does not dominate $Y$, and each component of $D^h$ dominates $Y$. $D^v$ (resp. $D^h$) is called the vertical (resp. horizontal) part of $D$. We say that $D$ is vertical (resp. horizontal) if $D=D^v$ ($D=D^h$). Assume that $D$ is vertical, then $D$ is called exceptional over $Y$ (or $f$-exceptional) if $\codim f(D)\geq2$.

Let $f$ be a birational map between projective normal varieties. We denote by $\Ex(f)$ the sum of all $f$-exceptional prime divisors. $f$ is a called a birational contraction if $\Ex(f^{-1})=0$.

\section{Quasi-Albanese morphisms}

\begin{defn}[Quasi-abelian variety]\label{def:Qabelian}
A quasi-abelian variety $G$ is a quasi-projective commutative algebraic group which is an extension of an abelian variety $A$ by an algebraic torus $T$: $0\to T\to G\xrightarrow{p} A\to 0$. We call $T$ the torus part and $ A$ the abelian part of $G$. Let \[M=\Hom(T,\Gg_m) \text{ , } N=\Hom(\Gg_m,T)\] be the character and cocharacter lattice of $T$. Let $m\in M$. Then $m$ induces a $T$-action on $\Aa^1$ given by $t\cdot x=m(t)\cdot x$ for $t\in T$, $x\in\Aa^1$. This $T$-action on $\Aa^1$ associates a line bundle $L_m=G\times^T\Aa^1$ on $A$. The pullback $p^*L_m$ has a canonical global nowhere-vanishing section $\chi^m$, whose restriction to each trivializing chart of $G\to A$ is just the usual $\chi^m$. The map $M\to \Pic(A)$, $m\mapsto L_m$, is a group homomorphism whose image is contained in $\Pic^0(A)$ \cite[Exposé VIII, Proposition 3.7]{SGA7-1}.

Let $\Sigma$ be a fan in $N_{\Rr}$ and let $X_{\Sigma}$ be the corresponding toric variety. The $T$-action on $X_{\Sigma}$ associates a $X_{\Sigma}$-bundle $G_{\Sigma}=G\times^TX_{\Sigma}\to A$. Then  $D_{\Sigma}=G_{\Sigma}\setminus G$ is a divisor and the pair $(G_{\Sigma},D_{\Sigma})$ is lc and Calabi-Yau: $K_{G_{\Sigma}}+D_{\Sigma}\sim0$. We call $(G_{\Sigma},D_{\Sigma})$ a toric partial compactification of $G$, and if $\Sigma$ is complete, then a toric compactification. If $\Sigma$ is a fan corresponding to $(\Pp^1)^d$ relative to a basis of $N$, then we call $(G_{\Sigma},D_{\Sigma})$ a standard compactification. 
\end{defn}

\begin{defn}[Quasi-Albanese morphism]

Let $X^0$ be a quasi-projective variety. The quasi-Albanese morphism $\alpha:X^0\to G$ is a morphism to a quasi-abelian variety $G$, such that for any morphism $\alpha':X^0\to G'$ to a quasi-abelian variety, there is a unique morphism $\beta:G\to G'$ such that $\beta\circ\alpha=\alpha'$.

Let $X^0$ be a smooth quasi-projective variety. Let $X^0\subset X$ be a compactification such that $D_X:=X\setminus X^0$ is an snc divisor. Then the quasi-Albanese variety $G=H^0(X,\Omega_X^1(\log D_X))^{\vee}/H_1(X^0,\Zz)$, and the quasi-Albanese morphism $\alpha:X^0\to G$ is constructed using the integral along paths of logarithmic one-forms \cite{iitaka1976logarithmic,fujino2024quasi}.

Let $X^0$ be a klt quasi-projective variety. Let $\pi:Y^0\to X^0$ be a resolution. By \cite{hacon2007shokurov,fujino2026rationalchainconnectedness} and the rigidity Lemma \cite[5.3 Proposition]{KollarRationalCurves}, the quasi-Albanese morphism of $Y^0$ induces a morphism from $X^0$ to a quasi-abelian variety. This morphism is the quasi-Albanese morphism of the klt variety $X^0$.    
\end{defn}

\begin{thm}[{\cite[Theorem 1.2]{zhu2025quasi}}]\label{thm:zhu1}
Let $(X,D_X)$ be a projective lc pair with $\kappa(K_X+D_X)=0$, let $X^0$ be its klt locus, and let
$\alpha:X^0 \to G$ be the quasi-Albanese morphism. Then $\alpha$ is a fibration.
\end{thm}

\begin{thm}[{\cite[Theorem 3.2]{zhu2025quasi}}]\label{thm:zhu2}
Let $G$ be a quasi-abelian variety and let $(\Pp_A,H)$ be a standard compactification of $G$. Let $(X,D_X)$ be a projective lc Calabi-Yau pair, let $\pi:Y\to X$ be a log resolution, let $D_Y=\pi^{-1}_*D_X+\Ex(\pi)$. Assume that there is a fibration $f:Y\to \Pp_A$ with $f^{-1}H\subset D_Y^{=1}$. Then $(Y,D_Y)$ has a good minimal model $(Z,D_Z)$ over $\Pp_A$ with $K_Z+D_Z\sim_{\Qq}0$.
\end{thm}

\begin{thm}\label{thm:zhu3}
Let $(X,D_X)$ be a projective lc Calabi-Yau pair, let $X^0$ be its klt locus, and let
$\alpha:X^0 \to G$ be a fibration to a quasi-abelian variety $G$. Let $T$ (resp. $A$) be the torus part (resp. abelian part) of $G$. Let $M$ (resp. $N$) be the character lattice (resp. cocharacter lattice) of $T$. Let $(\Pp_A,H)$ be a standard compactification of $G$. Then the following holds.
\begin{enumerate}
    \item There is a crepant birational contraction $\mu:(Z,D_Z=\mu^{-1}_*D_X+\Ex(\mu))\dashrightarrow(X,D_X)$ from a $\Qq$-factorial dlt Calabi-Yau pair $(Z,D_Z)$, such that the rational map $g:(Z,D_Z)\to \Pp_A$ is a morphism.
    
    \item For the previous given $g:(Z,D_Z)\to\Pp_A$, the following are equivalent.
    \begin{enumerate}
        \item $\alpha$ is proper. 
        \item $\alpha$ is proper in codimension one.
        \item $D_Z^{=1}$ is $g$-vertical.
   \end{enumerate} 
    
     \item If $\alpha$ is proper, then the following hold. 
     \begin{enumerate}
     \item $(X,D_X)$ is toroidal-dlt.
     \item There is a PL-map $\Sigma(\alpha):\Sigma(X,D_X)\to N_{\Rr}=\Hom(M,\Rr)$.
    \item $\Sigma(\alpha):\Sigma(X,D_X)\to N_{\Rr}$ is bijective. In particular, $\bigcup_S \Sigma(\alpha)(\sigma_S)$ is a complete fan.
    \item Let $(W,D_W)$ be the toric compactification of $G$ associated to this fan. Then $\alpha$ extends to a morphism $h:X\to W$. 
    \item For the previously given $h$, there is finite cover $V\to W$ which is \'etale over $W\setminus D_W$ such that $(X,D_X)\times_WV\cong (V,D_V)\times (F,D_F)$ over $V$. In particular, the toric compactification $(W,D_W)$ is projective. Here, the fiber product means the normalized fiber product.
    \end{enumerate}
\end{enumerate}
\end{thm}

We define toroidal-dlt pairs, the monomial valuation space $\Sigma(X,D_X)$, and prove this theorem for lc pairs in Section 3. Now we prove the theorem for dlt pairs. 

\begin{proof}[Proof for dlt pairs]
(1) Let $(\Pp_A,H)$ be a standard compactification of $G$. Let $\pi:Y\to X$ be a log resolution such that $D_{Y}':=\pi^{-1}(X\setminus X^0)$ is an snc divisor, and $f:Y\to \Pp_A$ is a morphism. 

Let $D_Y:=\pi^{-1}_*D_X+\Ex(\pi)$ and write \[K_Y+D_Y=\pi^*(K_X+D_X)+E_Y\] for some $\pi$-exceptional divisor $E_Y\geq0$. Then \[f^{-1}H\subset D'_{Y}\subset D_Y^{=1}.\] By Theorem \ref{thm:zhu2}, there is a good minimal model $(Z,D_Z)$ of $(Y,D_Y)$ over $\Pp_A$ with $K_{Z}+D_{Z}\sim_{\Qq}0$. Let $g:Z\to\Pp_A$ be the induced fibration, let $D_{Z}'$ be the pushforward of $D_Y'$, and let $\mu:Z\dashrightarrow X$ be the induced birational map. Since the MMP $Y\dashrightarrow Z$ contracts precisely the divisor $E_Y$ and $E_Y\subset \Ex(\pi)$, we have $\mu:Z\dashrightarrow X$ is a birational contraction. 

Let $Q$ be a component of $D_Y^{=1}-D_Y'$. Then $f(Q)$ intersects $X^0$. Since $X^0$ is the klt locus of $(X,D_X)$, we have  $-1<a(Q,X,D_X)=a(Q,Y,D_Y-E_Y)=-1+v_Q(E_Y)$. Hence $Q$ is a component of $E_Y$. Hence \[ \Supp(D_Y^{=1}-D_Y')\subset \Supp E_Y.\] Since the MMP $Y\dashrightarrow Z$ contracts $E_Y$, we have \[D_{Z}'=D_{Z}^{=1}.\]

(2) $(a)\Rightarrow(b)$ This is clear.

$(b)\Rightarrow(c)$. Assume that $\alpha$ is proper in codimension one. Then $\alpha\circ \pi:Y\setminus D_Y'\to G$ is proper in codimension one. Hence $D_{Y}'-f^{-1}H$ is $f$-exceptional. Hence $D_{Z}'-g^{-1}H$ is $g$-exceptional. Since $g$ is flat over $G$ (Proposition \ref{prop:tdlt-morphism} (2)), $D'_{Z}=g^{-1}H$. Hence, $D_{Z}^{=1}=D_{Z}'=g^{-1}H$ is $g$-vertical.

$(c)\Rightarrow(a)$. Assume that $D_{Z}^{=1}$ is $g$-vertical. Then $D_{Z}^{=1}=g^{-1}H$ (Proposition \ref{prop:tdlt-morphism} (3)). To show that $\alpha$ is proper, it suffices to show that $D_{Y}'=f^{-1}H$.

Let $Q$ be a component of $D_{Y}'$. Then $\pi(Q)\subset D_X^{=1}$. Let $D_i$ be components of $D_X^{=1}$ containing $\pi(Q)$. Let $S$ be the component of $\cap_iD_i$ containing $\pi(Q)$ and let $v=\sum_iv_Q(D_i)v_{D_i}\in \sigma_S$ be the weight $v_Q(D_i)$ monomial valuation centered at the generic point of $S$. In other words, $v$ is the retraction of $v_Q$ \cite[4.3]{jonsson2012valuations}.
 
We use the notation in Definition \ref{def:Qabelian}.
In particular, there is a group homomorphism $M\to \Pic^0(A)$, $m\mapsto L_m$, and there are nowhere-vanishing sections $\chi^m\in H^0(G,p^*L_m)$. Since $(X,D_X)$ is dlt, the morphism $Y\to A$ factors through a morphism $\beta:X\to A$. Hence we may view the pullback of $\chi^m$ to $X$ as a rational section of $\beta^*L_m$ on $X$.

Since $v$ is an lc place and $D_{Z}^{=1}=g^{-1}H$, the center of $v$ on $\Pp_A$ is contained in $H$. Hence, $v(\chi^{m})\neq0$ for some $m\in M$. Since $\di_X(\chi^m)$ is supported in $D_X^{=1}$, we have $v_Q(\chi^{m})=v(\chi^m)\neq0$ \cite[Lemma 4.7]{jonsson2012valuations}. Hence the center of $Q$ on $\Pp_A$ is contained in $H$.

(3) (a) Since $(X,D_X)$ is dlt, this automatically holds.

(b)  Define \[\Sigma(\alpha):\Sigma(X,D_X)\to N_{\Rr},v\mapsto(m\mapsto v(\chi^m)).\] Since $v(\chi^{m_1+m_2})=v(\chi^{m_1})+v(\chi^{m_2})$, this map is well-defined. Since  $\di_X(\chi^m)$ is supported on $D_X^{=1}$, $\Sigma(\alpha)$ is linear on cones in $\Sigma(X,D_X)$.

(c) By proposition \ref{prop:vertical}(1), $\Sigma(g):\Sigma(Z,D_Z)\to\Sigma(\Pp_A,H)$ is a subdivision. In particular, it is bijective. Since $|\Sigma(\Pp_A,H)|=N_{\Rr}$ and $(X,D_X)$, $(Z,D_Z)$ are crepant birational, $\Sigma(\alpha):\Sigma(X,D_X)\to N_{\Rr}$ is bijective. 

(d) Let $x\in X\setminus X^0=D_X^{=1}$ be a point and let $D_i$ be components of $D_X^{=1}$ containing $x$. Let $S$ be the component of $\cap_iD_i$ containing $x$. Then $\sigma_S=\Cone(v_{D_1},\cdots,v_{D_r})$. Let $\tau_S=\Sigma(\alpha)(\sigma_S)$. Let $m\in \tau_S^{\vee}\cap M$. Then $v_{D_i}(\chi^m)\geq0$ for any $i$. If $Q$ is a prime divisor containing $x$ and intersecting $X^0$, then $v_Q(\chi^m)=0$. Hence the morphism \[f:\Spec \Oo_{X,x}\to\Spec k [\tau_S^{\vee}\cap M]\]
extends to codimension one points. Since $X$ is normal, $f$ extends near $x$.

(e) This follows from Proposition \ref{prop:vertical}(2).
\end{proof}

\begin{exa}\label{exa1}Consider the log smooth pair $(\Pp^2,H)$, where $H$ is a sum of three lines.
Let $\pi:X \to \Pp^2$ be the blowup at a point $P\not\in H$. Let $E$ be the exceptional curve and let $D_X=\pi^{-1}_*H+E$. Then $\kappa(K_X+D_X)=0$,  and the quasi-Albanese morphism $X\setminus D_X\cong (\Cc^{\times})^2\setminus \{P\} \to (\Cc^{\times})^2$ is proper in codimension one, but not proper.    
\end{exa}

\begin{exa}\label{exa2}
Consider the log smooth pair $(\Pp^2,L+Q)$, where $L=(l=0)$ is a line and $Q=(q=0)$ is a conic. Then the quasi-Albanese morphism $\alpha:\Pp^2\setminus L+Q\to \Cc^{\times}$ is $x\mapsto l^2/q$. The induced rational map $\Pp^2\dashrightarrow\Pp^1$ is given by the linear system $\la 2L, Q \ra$.

We blow up to resolve the base locus of $\la 2L, Q \ra$. To simplify notation, we will use the same letter to denote both a divisor and its strict transform. Let $X\to \Pp^2$ be the blowup at $L\cap Q$, and let $E_1$, $E_2$ be the exceptional curves. Let $Y\to X$ be the blowup at $E_1\cap Q$, $E_2\cap Q$, and let $E_3$, $E_4$ be the exceptional curves. Let $\pi:Y\to\Pp^2$ be the composed birational morphism and let $f:Y\dashrightarrow\Pp^1$ be the induced rational map. Then $\pi^*\la2L,Q\ra=E_1+E_2+2E_3+2E_4+\la 2L+E_1+E_2,Q \ra$ and $\la 2L+E_1+E_2,Q \ra$ is base point free, i.e., $f$ is morphism. $(Y,L+Q+\sum_iE_i)$ is Calabi-Yau, $E_3$, $E_4$ are sections of the fibration $f$, $f^*0=2L+E_1+E_2$, $f^*\infty=Q$, and the dual complex of $L+Q+\sum_iE_i$ is a circle $Q-E_3-E_1-L-E_2-E_4-Q$.
\end{exa}

\section{Toroidal pairs}

\subsection{Toroidal pairs}

\begin{defn}[Valuations]
Let $X$ be a normal variety. 

A valuation $v$ is a map $v:K(X)^{\times}\to\Rr$ satisfying $v(fg)=v(f)+v(g)$, $v(f+g)\geq\min\{v(f),v(g)\}$, and $v(k^{\times})=0$. We say $v$ has a center on $X$ if there is a point $x\in X$ such that there is a local injective morphism $\Oo_{X,x}\to \Oo_v:=\{f\in K(X)\mid v(f)\geq0\}$ between local rings. If $v$ has a center on $X$, we denote by $C_X(v)$ its center. Define $\Val_X$ as the set of valuations having center on $X$. 

If $P$ is a prime divisor over $X$, then the order of vanishing along $P$ defines a valuation $v_P$. A valuation $v\in\Val_X$ is called divisorial if $v=c\cdot v_P$ for some $c\in\Rr_{>0}$ and some prime divisor $P$ over $X$.

Let $D$ be a $\Qq$-Cartier divisor. If $mD$ is Cartier and $s\in K(X)^{\times}$ is a local defining function of $mD$, we define  $v(D):=\frac{1}{m}v(s)$.
\end{defn}

\begin{defn}[Toroidal pairs \cite{KKMSD,WeakSemistableReduction2000}]\label{defn:toroidal}
Let $X$ be a normal variety and let $D_X\geq0$ be a reduced divisor. 

We call $(X,D_X)$ toroidal at a (closed) point $x\in X$ if there is an affine toric variety $X_{\sigma}$ with distinguished point $x_{\sigma}$, and an isomorphism of complete local $k$-algebras $\widehat{\Oo}_{X,x}\cong \widehat{\Oo}_{X_{\sigma},x_{\sigma}}$, which identifies $D_X$ with the toric boundary $D_{\sigma}$.
We call $(X,D_X)$ toroidal if it is toroidal at any point of $X$. Let $D_i$, $i\in I$, be components of $D_X$. We say that $(X,D_X)$ has no self-intersection if the $D_i$ are normal. In this note, we only consider toroidal pairs without self-intersection.

The following fact is important. Let $(X_{\sigma},D_{\sigma})$ be a local model of a toroidal pair $(X,D_X)$ at $x$. If $(X,D_X)$ has no self-intersection, then there is a Zariski (!) open neighborhood $V$ of $x$, and an \'etale morphism $p:V\to X_{\sigma}$ such that $p^{-1}D_{\sigma}=D_X|_{V}$. The proof is outlined  as a footnote in \cite[Page 195]{KKMSD}, and a detailed proof is given by Qu \cite[Theorem 2.13]{qu2025saturated}. 

Assume from now on that $(X,D_X)$ is a toroidal pair without self-intersection.

The components of $\cap_{i\in J}D_i\setminus \cup_{i \in I\setminus J}D_i$ form a stratification of $X$. Each stratum corresponds to an open orbit in local models. 

From the local description, $X$ has rational singularities, $K_X+D_X$ is Cartier, and $(X,D_X)$ is lc. The lc centers of $(X,D_X)$ are normal and are precisely the closures of strata. 

Let $S^0$ be a stratum, let $S$ be its closure, and let $D_S=\Diff^*_SD_X$, then $\Diff^*_SD_X=S\setminus S^0$ and $(S,D_S)$ is toroidal without self-intersection. Let $\eta_S$ be the generic point of $S$. Define
\begin{enumerate}
    \item $M_S:=$ the lattice of Cartier divisors on $\Spec \Oo_{X,\eta_S}$ supported on $D_X$, 
    \item $N_S:=\Hom(M_S,\Zz)$,
    \item $M_S^{+}:=$ the monoid of effective elements in $M_S$, and
    \item $\sigma_S:=\Hom(M_S^{+},\Rr_{\geq0})\subset N_{S,\Rr}$ the dual of $M_S^{+}$.
\end{enumerate}
Then $\sigma_S$ is a full-dimensional strictly convex rational polyhedral cone in $N_{S,\Rr}$, and
the local model of $(X,D_X)$ at $\eta_S$ is the toric pair corresponding to $\sigma_S\subset N_{S,\Rr}$ over $k(\eta_S)$. In particular, \[\codim_XS=\dim \Oo_{X,\eta_S}=\rank M_S.\] 

Let $v\in \sigma_S$. Then $v$ determines a valuation described as follows. Let $D_i$, $1\leq i\leq r$, be components of $D_X$ containing $S$. Then 
\[M_S=\{m=(m_1,\cdots,m_r)\in\Zz^r\mid\sum_im_iD_i\text{ is Cartier at }\eta_S\}\] For $m\in M_S$, let $\chi^m\in K(X)^{\times}$ be the unique (up to multiplication by a unit in $\Oo_{X,\eta_S}$) function such that $\di_X(\chi^m)=\sum_im_iD_i$. Let $f\in \Oo_{X,\eta_S}$. Since $\widehat{\Oo}_{X,\eta_S}\cong k(\eta_S)[[M_S^{+}]]$, $f$ can be written as $\sum_{m\in M_S^{+}}c_m\chi^m$. Define \[v(f):=\min_{c_m\neq0}\la v,m\ra.\] Hence $v(\chi^m)=\la v, m\ra$. Let $u_i\in N_S$, $1\leq i\leq r$, be elements such that $\la u_i,m\ra=m_i$. Then $\sigma_S=\Cone(u_1,\cdots,u_r)$ and the valuation determined by $u_i$ is just the divisorial valuation $v_{D_i}$.

The cone $\sigma_S$ of all strata $S^0$ form a rational polyhedral cone complex $\Sigma(X,D_X)=\{(\sigma_S,N_{S})\}$.
\end{defn}

\begin{defn}
A rational polyhedral cone complex (cone complex for short) $\Sigma=\{(\sigma,N_{\sigma})\}$ consists of a finite
collection of lattices $N_{\sigma}$ and full-dimensional strictly convex rational polyhedral cones  $\sigma\subset N_{\sigma,\Rr}$. The cones $\sigma$ are glued together to form the space $|\Sigma|$ with the following compatible conditions. 
1. If $\sigma\in\Sigma$ is a cone and $\sigma'$ is a face of $\sigma$, then $\sigma'$ is also in $\Sigma$ and $N_{\sigma'}=N_{\sigma}|_{\Span(\sigma')}$. 2. The intersection of two cones is a union of faces of both of them.

A morphism $f:\Sigma=\{(\sigma,N_{\sigma})\}\to\Sigma'=\{(\tau,N_{\tau})\}$ of cone complexes is a collection of linear maps $f_{\sigma}:(\sigma,N_{\sigma})\to (\tau,N_{\tau})$
with the compatible condition that if $\sigma'$ is a face of $\sigma$ then $f_{\sigma'}$ is the restriction of $f_{\sigma}$.  
\end{defn}

\begin{rem}
If we only consider the real vector space $N_{\sigma,\Rr}$, then we will call $\Sigma$ the real cone complex. For example, consider $f:(\Aa^1,o)\to(\Aa^1,o), x\mapsto x^k$. Let $v_o$ be the divisorial valuation of the origin $o$. Then $\Sigma(f)(v_o)=k v_o$, $\Sigma(f)$ is not an isomorphism of cone complexes, but an isomorphism of real cone complexes.    
\end{rem}

\begin{exa}
Let $(X_{\Sigma},D_{{\Sigma}})$ be the toric pair corresponding to a fan $\Sigma\subset N_{\Rr}$. Then $(X_{\Sigma},D_{{\Sigma}})$ is toroidal without self-intersection. For each cone $\sigma \in\Sigma$, let $V(\sigma)$ be the orbit closure. Then $M_{V(\sigma)}=M/\sigma^{\perp}\cap M$, $N_{V(\sigma)}=N|_{\Span(\sigma)}$ and $\{\sigma,N_{V(\sigma)}\}$ for all cone $\sigma$ recovers the fan $\Sigma$. Let $v\in\Sigma$, then the valuation determined by $v$ is given by the inner product between $M$ and $N$, i.e., $v(\chi^m)=\la v, m\ra$ for $m\in M$. If $v$ is contained in the relative interior of a cone $\sigma$, then $\Oo_{X_{\Sigma},\eta_{V(\sigma)}}\subset \Oo_{v}$, i.e.,
the center of $v$ on $X_{\Sigma}$ is the generic point of $V(\sigma)$.
\end{exa}

The following series of lemmas are useful. Lemma \ref{lem:toroidal} (2) and Proposition \ref{prop:tdlt-basic} (5), (6) are not actually used in the proof of the main theorems, but we include them for the sake of completeness.

\begin{lem}\label{lem:toric}
Let $(X,D_X)=(G_{\Sigma},D_{\Sigma})$ be a toric partial compactification of a quasi-abelian variety $G$. Let $A$ and $T$ be the abelian and torus parts of $G$, respectively.
\begin{enumerate}
    \item Let $f:Y\to X$ be a finite cover which is \'etale over $X\setminus D_X$, then $(Y,D_Y:=f^{-1}D_X)$ is a toric partial compactification of a quasi-abelian variety.
    \item Let $f:Y\to X$ be a projective birational morphism. Let $D_Y=f^{-1}_*D_X+\Ex(f)$. Assume that $K_Y+D_Y=f^*(K_X+D_X)$. Let $X_{\Sigma}$ be the toric variety corresponding to the fan $\Sigma\subset N_{\Rr}$. Then there is a subdivision $\Sigma'$ of $\Sigma$ such that $f$ is the morphism induced by the toric morphism $h:X_{\Sigma'}\to X_{\Sigma}$, i.e., $f=G\times^{T} h$.
    \end{enumerate}
\end{lem}
\begin{proof}

(1) Since $G':=Y\setminus D_Y\to G$ is finite \'etale, $G'$ is quasi-abelian and $q:=f|_{G'}:G'\to G$ is an isogeny \cite[Theorem 4.2]{fujino2024quasi}. The isogeny $q:G'\to G$ induces isogenies between the tori $q_T:T'\to T$ and the abelian varieties $q_A:A'\to A$. $q_T$ induces a finite-index inclusion between the cocharacter lattices $N'\subset N$.

Let $X_{\Sigma,N'}$ be the toric variety corresponding to the fan $\Sigma\subset N'_{\Rr}$. Let $\sigma$ be a cone in $\Sigma$. Since $M\subset M'$ is of finite index, there is a positive integer $l$, such that for any $m\in k[\sigma^{\vee}\cap M']$, $(\chi^m)^l\in k[\sigma^{\vee}\cap M]$. Hence the morphism $h:X_{\Sigma,N'}\to X_{\Sigma,N}$ is finite.

Consider $F:=q\times^{q_T}h:G'_{\Sigma,N'}=G'\times^{T'}X_{\Sigma,N'}\to X=G_{\Sigma,N}=G\times^{T}X_{\Sigma,N}$. $F$ factors as $G'_{\Sigma,N'}\to X\times_AA'\to X$. Since the restriction of the first morphism to a trivializing chart $U\subset A'$ of $G'\to A'$ is $\id_{U'}\times h$, the first morphism is finite. Since $q_A$ is finite, the second morphism is finite. Hence $F$ is finite. Then both $G'_{\Sigma,N'}$ and $Y$ are normalization of $X$ in $K(G')$. Hence $G'_{\Sigma,N'}=Y$.

(2) See also \cite[Proof of Lemma 2.3.2]{Brown2018GeomToric}. Let $E\subset D_Y$ be an exceptional divisor. Since $a(E,Y,D_Y)=-1$ and $X\setminus D_X$ is smooth, $f(E)\subset D_X$. Since $f^{-1}D_X$ is of pure codimension one, we have $D_Y=f^{-1}D_X$ and $f$ is isomorphic over $G$.

Let $L$ be an $f$-ample Cartier divisor. By Lemma \ref{lem:pic} (1), we may assume that $\Supp L\subset D_Y$. Then components of $L$ correspond to monomial valuations of $(X,D_X)$, i.e, for component $Q$ of $L$, we have $v_Q\in \Sigma(X,D_X)=\Sigma$. Since the $v_Q$ are $G$-invariant, we have $\bigoplus_{d\geq0}f_*\Oo_Y(dL)$ is a $G$-invariant graded $\Oo_X$-algebra and $f:Y=\Proj_X\bigoplus_{d\geq0}f_*\Oo_Y(dL)\to X$ is $G$-invariant. Let $Y_0$ and $X_0$ be the fibers of $Y$ and $X$ over $0\in A$ and let $f_0:=f|_{Y_0}$. Then $f_0$ is $T$-invariant, $Y=G\times^T Y_0$ and $f=G\times^T f_0$. 

Let $U\subset A$ be a trivializing chart of $G\to A$, then $Y|_{U}\cong Y_0\times U$. Hence $Y_0$ is integral and normal. Since $f$ is isomorphic over $G$, $f_0$ is isomorphic over $T\subset X_0$. Hence $Y_0$ and $f_0$ are toric. Since $f_0$ is birational, $Y_0=X_{\Sigma'}$ for some subdivision $\Sigma'$ of $\Sigma$. 
\end{proof}

\begin{lem}\label{lem:pic}\leavevmode
\begin{enumerate}
\item Let $(X,D_X)$ be a toric partial compactification of a quasi-abelian variety $G$. Let $D_i$, $i\in I$, be components of $D_X$. Then $\bigoplus_{i\in I}\Zz[D_i]\oplus \Cl(A)\to \Cl(X)$ is surjective.
\item Let $(X,D_X)$ be a toroidal pair. Let $x\in X$ be a closed point. Let $D_i$, $i\in I$, be components of $D_X$. Then $\bigoplus_{i\in I}\Zz[D_i]\to \Cl(\Oo_{X,x})$ is surjective.
\end{enumerate}
\end{lem}
\begin{proof}
(1) Let $L$ be a divisor on $X$. Let $Y\to A$ be an $\Aa^d$-bundle which is a partial compactification of $G$. Let $\overline{L|_{G}}$ be the closure of $L|_{G}$ in $Y$. Then there is a rational function $h$ and a divisor $L_A$ on $A$ such that $\overline{L|_{G}}=p^*L_A+\di_X(h)$ on $Y$. Hence $L\sim p^*L_A+$ (some divisor supported on $D_X$). 

(2) This follows from $\Cl(\Oo_{X,x})\to\Cl(\widehat{\Oo}_{X,x})$ is injective \cite[Theorem 6.5]{samuel1964lectures} and $\bigoplus_{i\in I}\Zz[\widehat{D}_i]\to \Cl(\widehat{\Oo}_{X,x})$ is surjective \cite[Theorem 5.8]{ishiro2025local}.
\end{proof}

\begin{lem}\label{lem:toroidal}
Let $(X,D_X)$ be a toroidal pair without self-intersection.     
\begin{enumerate}

\item Let $f:Y \to X$ be a finite cover which is \'etale over $X\setminus D_X$. Then $(Y,D_Y:=f^{-1}D_X)$ is toroidal without self-intersection and $f$ is a toroidal morphism. 

\item Let $f:Y\to X$ be a projective birational morphism. Let $D_Y=f^{-1}_*D_X+\Ex(f)$. Assume that $K_Y+D_Y=f^*(K_X+D_X)$. Then $(Y,D_Y)$ is toroidal without self-intersection and $f$ is a toroidal morphism.

\end{enumerate}
\end{lem}
\begin{proof}

(1) The proof is relatively easy to work with analytic spaces and then will be the same as that of Abhyankar's lemma (see e.g. \cite[Proof of Proposition 2.25]{enokizono2024semistable}). 

We assume that $(x\in X,D_X)$ is analytically isomorphic to $((x_{\sigma}\times 1)\in X_{\sigma}\times (\Cc^{\times})^{r},D_{\sigma}\times (\Cc^{\times})^{r})$, where $\sigma\subset N_{\Rr}$ is of full dimensional. Let $V\subset X$ be a sufficiently small analytically open neighborhood of $x$. Then $\pi_1(V\setminus D_X)$ is  canonically isomorphic to $N$. Let $U$ be a connected component of $f^{-1}V$. Since $U\setminus D_Y\to V\setminus D_X$ is \'etale, it is induced by a finite-index subgroup $N'\subset N$. 

Let $X_{\sigma,N'}$ be the toric variety corresponding to $\sigma\subset N'_{\Rr}$ and let $f_{\sigma}:X_{\sigma,N'}\to X_{\sigma}$ be the finite cover. Since the corresponding cover between tori $T_{N'}\to T_N$ is also induced by $N'\subset N$, $(f_{\sigma}\times\id)^{-1}(V\setminus D_X)\cong U\setminus D_Y$. Then both $U\to V$ and $(f_{\sigma}\times \id)^{-1}V\to V$ are finite covers extending the cover $U\setminus D_Y \to V\setminus D_X$. Hence $U\cong (f_{\sigma}\times \id)^{-1}V$, $f|_{U}\cong f_{\sigma}\times \id$ \cite[Lemma 2.24]{enokizono2024semistable}. Hence $(Y,D_Y)$ and $f$ are toroidal. Since on $U$, the components of $D_Y$ and the components of $D_X$ are in one-to-one correspondence, the components of $D_Y$ are normal. Hence, $(Y, D_Y)$ has no self-intersection.

(2) Let $x\in X$ and let $X_{\sigma}\times\Gg_m^r$, $X_{\sigma}=\Spec k[M^+]$, be a local model at $x$, where $\sigma\subset N_{\Rr}$ is of full dimensional. Then there is an open neighborhood $V\ni x$ and an \'etale morphism $p:V\to X_{\sigma}\times\Gg_m^r$. Such a neighborhood exists if $(X,D_X)$ has no self-intersection (see \cite[Page 195]{KKMSD} or \cite[Theorem 2.13]{qu2025saturated}).

As in the proof of Lemma \ref{lem:toric} (2), we have $D_Y=f^{-1}D_X$ and $f$ is isomorphic over $X\setminus D_X$. Let $L$ be an $f$-ample Cartier divisor. By Lemma \ref{lem:pic}(2), we may assume that $L$ is supported on $D_Y$ over $x$. By replacing $L$ by $L-\di_Y(\chi^m)$ for some sufficiently large $m$, we may assume that $L$ is negative over $x$. Then $(f_*\Oo_Y(dL))_{x}\subset \Oo_{X,x}$. Let $v_j$ be all valuations corresponding to components of $D_Y$ over $x$. Since $f$ is crepant, the $v_j$ are monomial valuations of $(X,D_X)$. Since $L$ is negative over $x$, $a_j:=v_j(L)\leq0$. Let $S_d=\{m\in M\mid \la v_j,m\ra+da_j\geq0 \ \forall j\}\subset M^+$ and let $S=\bigoplus_{d\geq 0}S_d$. Then $k[S]=\bigoplus_{d\geq 0}k[S_d]$ is a graded family of ideals in $k[M^+]$.

We show that $\bigoplus_{d\geq0}(f_*\Oo_Y(dL))_{x}\cong k[S]\otimes_{k[M^+]}\Oo_{X,x}$. Since $k[S_d]\Oo_{X,x}\subset (f_*\Oo_Y(dL))_{x}$ is clear, we show the converse. Let $h\in (f_*\Oo_Y(dL))_{x}$. Then $v_j(h)+da_j\geq0$ for all $j$. Since $\widehat{\Oo}_{X,x}\cong k[[M^{+}]][[z_1,\cdots,z_r]]$, we can write $h=\sum h_m\chi^m$, $h_m\in k[[z_1,\cdots,z_r]]$. Then $\la v_j,m\ra+da_j\geq0$ for all $j$ and all $m$ with $h_m\neq0$. Hence $h\in k[S_d]\widehat{\Oo}_{X,x}$. Since ${\Oo}_{X,x}\to \widehat{\Oo}_{X,x}$ is faithfully flat \cite[Lemma 10.97.3]{stacks-project}, $(f_*\Oo_Y(dL))_{x}= {\Oo}_{X,x}\cap k[S_d]\widehat{\Oo}_{X,x}=k[S_d]\Oo_{X,x}$.

Let $Y_{\sigma}=\Proj k[S]$ and let $f_{\sigma}: Y_{\sigma}\to X_{\sigma}$ be the projection. Then $Y_{\sigma} $ and $f_{\sigma}$ are toric \cite[Chapter 11]{CoxLittleSchenckToric}. Let $X_{x}:=\Spec \Oo_{X,x}$. Then
\[Y_{x}:=Y\times_XX_x\cong\Proj_{X_x}\bigoplus_{d\geq 0}(f_*\Oo_Y(dL))_x\cong \Proj_{X_x} k[S]\otimes_{k[M^+]}\Oo_{X,x}\cong Y_{\sigma}\times_{X_{\sigma}}X_x.\] After shrinking $V$, we have $Y_V\cong Y_{\sigma}\times _{X_{\sigma}}V$ and $f_V\cong f_{\sigma}\times _{X_{\sigma}}V$. Hence, the following diagram is a fiber product.
\begin{center}
\begin{tikzcd}
Y_V \arrow[d, "f_V"] \arrow[r] & Y_{\sigma}\times\Gg_m^r \arrow[d, "f_{\sigma}\times\id"] \\
V \arrow[r, "p"]               & X_{\sigma}\times\Gg_m^r 
\end{tikzcd}    
\end{center}
In particular, $f$ is toroidal, and since $p$ is \'etale, $(Y,D_Y)$ is toroidal without self-intersection.
\end{proof}

\subsection{Toroidal-dlt pairs}

\begin{defn}[Toroidal-dlt pairs]
Let $(X,D_X)$ be an lc pair. We call $(X,D_X)$ toroidal-dlt (tdlt for short) if there is an open subset $U\subset X$ such that 
\begin{enumerate}
    \item all lc centers intersect $U$, and
    \item $(U,D_X|_{U}=D_X^{=1}|_{U})$ is toroidal without self-intersection.
\end{enumerate}

Let $(X,D_X)$ be a tdlt pair and $U$ be its toroidal locus. Define $M_S:=M_{S\cap U}$ (and similar for $N_S$, $M_S^{+}$, $\sigma_S$) for an lc center $S$, and define $\Sigma(X,D_X):=\Sigma(U,D_{X}|_U)$. Dlt pairs (resp. qdlt pairs) are tdlt pairs, and are precisely those for which $\Sigma(X,D_X)$ consists of smooth cones (resp. simplicial cones). If $(X,D_X)$ is dlt, then $\Sigma(X,D_X)$ is the usual monomial valuation space $\text{QM}(X,D_X)$ \cite{jonsson2012valuations}, so is the cone of the dual complex $\Dd(D_X)$. 

Define $\LC(X,D_X):=\{v\in\Val_X\mid a(v,X,D_X)=-1\}$. Then $\LC(X,D_X)=|\Sigma(X,D_X)|$ as subsets of $\Val_X$. Indeed, let $\pi:(U',D_U')\to (U,D_U=D_X|_{U})$ be a toroidal log resolution. Under the identification $\Val_U=\Val_{U'}$, we have \[\LC(X,D_X)=\LC(U,D_U)=\LC(U',D_{U'})=|\Sigma(U',D_{U'})|=|\Sigma(U,D_{U})|,\]
where the first ``$=$'' due to any lc center intersects $U$, section ``$=$'' due to $\pi$ is crepant, third ``$=$'' is the usual LC=QM for log smooth pairs \cite[Proof of Lemma 2.3]{BlumLiuXu2022Open}, and the fourth ``$=$'' follows since $\pi$ is toroidal.
\end{defn}

The following theorem is a generalization of \cite[Proposition 34]{de2017dual}, and is also a special case of Moraga-Savaldi's criterion
\cite{moraga2025geometric}. We provide a proof using Koll\'ar's $\Pp^1$-linking theory \cite[Section 4.4]{Kollar2013sing} and Bergner's criterion \cite{bergner2021lipman}.

\begin{thm}[{\cite{moraga2025geometric}}]\label{thm:tdlt-criterion}
Let $(X,D_X)$ be an lc pair and let $S$ be an lc center. Let $M_S$ be the lattice of Cartier divisors on $\Spec\Oo_{X,\eta_S}$ supported on $D_X^{=1}$. Then \[\rank M_S\leq \codim_XS\] with equality holds iff at general points of $S$, $D_X=D_X^{=1}$ and $(X,D_X)$ is toroidal without self-intersection. 
\end{thm}

\begin{thm}\cite[Proposition 5.8, Theorem 5.14]{bergner2021lipman}\label{thm:bergner}
Let $X$ be a normal variety and let $D_X\geq0$ be a reduced divisor. Assume that $\Omega^{[1]}_X(\log D_X)$ is locally free and locally generated by closed logarithmic one-forms. Then $(X, D_X)$ is toroidal without self-intersection.   
\end{thm}

\begin{proof}[Proof of Theorem \ref{thm:tdlt-criterion}]
The if part follows from the properties of toroidal pairs. We prove the only-if part. 

We first treat the case when $S$ is a closed point. Let $\pi:(Y,D_Y)\to(X,D_X)$ be a dlt modification. Let $T\subset Y$ be a minimal lc center dominating $S$. Let $\phi\in K(X)^{\times}$ be a function such that $\Supp\di_X(\phi)\subset D_X^{=1}$ near $S$. Then $\Supp\di_Y(\pi^*\phi) \subset D_Y^{=1}$ near the generic point of $T$. Hence there is a homomorphism \begin{align}\label{equ:pullback}
\pi^*:M_S\to M_T, \ \phi\mapsto \pi^*\phi.    
\end{align}
We claim that this homomorphism is injective. Assume the contrary, i.e., there is a component $Q\supset S$ of $\di_X(\phi)$, but $\di_Y(\phi)=0$ near the generic points of all components of $D_Y^{=1}$ contain $T$. Since $\di_Y(\phi)$ is supported on $D_Y^{=1}$ near $T$, we have $\phi\in\Oo^{\times}_{Y,\eta_T}$. Let $T'\subset Y$ be another minimal lc center dominating $S$ which is directly $\Pp^1$-linked to $T$. Then there is an lc center $Z\subset Y$ containing $T$, $T'$ such that $\pi(Z)=S$ and $(Z,D_Z:=\Diff^*_ZD_Y)$ is crepant birational to a standard $\Pp^1$-link with $T$, $T'$ mapping to two sections. Since $v_T(\phi|_{Z})=0$, we have $v_{T'}(\phi|_{Z})=0$. Hence $\phi\in\Oo^{\times}_{Y,\eta_{T'}}$. 
Since all minimal lc centers dominating $S$ are $\Pp^1$-linked, $\phi$ is a unit near the generic points of all minimal lc centers dominating $S$. Let $Q_Y$ be the strict transform of $Q$ on $Y$. Then $\phi$ is not a unit near the generic point of the minimal lc center contained in $Q_Y$, which is a contradiction.

From the injective homomorphism \ref{equ:pullback}, we have \[\rank M_S\leq \rank M_T=\codim T\leq\dim X.\] 

Assume that $\rank M_S=\dim X$. Then $\rank M_S=\rank M_T=\codim T=\dim X$. We show that \[\alpha^*:M_S\otimes_{\Zz}\Oo_{X}\to\Omega_{X}^{[1]}(\log D_X^{=1}),\ \phi \mapsto d\log \phi,\] is isomorphic near $S$. Let $\phi_i\in K(X)^{\times}$, $1\leq i\leq n$, be rational functions representing a basis of $M_S$. Let \[\omega=d\log\phi_1\wedge\cdots\wedge d\log\phi_n.\] Then $\di_X(\omega)+D_X^{=1}\geq0$ near $S$. Since $(Y,D_Y)$ is log smooth near $T$, $M_T\otimes_{\Zz}\Oo_{Y}\to\Omega_{Y}^{[1]}(\log D_Y^{=1})$ is isomorphic and $D_Y=D_Y^{=1}$ near $T$. Hence $\bigwedge^{n}M_T\otimes_{\Zz}\Oo_{Y}\to\Oo_Y(K_Y+D_Y^{=1})$ is isomorphic near $T$. Since $\pi^*:M_S\to M_T$ is injective with finite index, $\pi^*\omega$ is a basis of $\Oo_Y(K_Y+D_Y^{=1})$ near $T$. Hence \[\di_Y(\pi^*\omega)+D_Y=\di_Y(\pi^*\omega)+D_Y^{=1}=0\text{ at } T.\] By $\di_Y(\pi^*\omega)+D_Y=\pi^*(\di_X(\omega)+D_X)$ and $\di_X(\omega)+D_X\geq\di_X(\omega)+D_X^{=1}\geq0$, we have \[\di_X(\omega)+D_X^{=1}=0\text{ and }D_X=D_X^{=1}\text{ at } S.\] Then $\omega$ is a basis of $\Oo_{X}(K_X+D_X^{=1})$ at $S$. Hence $\det(\alpha^*)$ is isomorphic near $S$. Since $M_S\otimes_{\Zz}\Oo_{X}$ and $\Omega_X^{[1]}(\log D_X^{=1})$ are reflexive, $\alpha^*$ is isomorphic near $S$. By Theorem \ref{thm:bergner}, $(X,D_X^{=1})$ is toroidal without self-intersection at $S$.  
 
Now assume that $s:=\dim S>0$. Let $x\in S$ be a general point and let $H_i$, $1\leq i\leq s$, be general ample Cartier divisors such that $x$ is an isolated point of $S\cap\bigcap_i H_i$, $(X,D_X+\sum_iH_i)$ is lc, and $x$ is an lc center. Then $M_x=M_S\oplus\Zz^s$. By the closed point case, we have $\rk M_S\leq \codim_X S$, and if the equality holds, then at $x$, $D_X=D_X^{=1}$ and $(X,D_X^{=1})$ is toroidal without self-intersection.
\end{proof}

We list basic properties of tdlt pairs. Note that (5) and (6) do not hold for dlt pairs. 

\begin{prop}\label{prop:tdlt-basic}Let $(X,D_X)$ be a tdlt pair and let $U$ be its toroidal locus.
\begin{enumerate}
    \item $(X,D_X)$ is klt (resp. plt) iff $D_X^{=1}=0$ (resp. components of $D_X^{=1}$ are disjoint).
    
    \item $X$ has rational singularities.
    
    \item Lc centers are normal and are precisely the closures of strata of $(U,D_X|_U)$ in $X$.

    \item Let $f:Y \to X$ be a finite cover which is quasi-\'etale over $X\setminus D_X^{=1}$. Let $D_Y=f^{-1}D_X^{=1}+f^*D_X^{<1}$. Then $(Y,D_Y)$ is tdlt. If further $(X,D_X)$ is qdlt, then $(Y,D_Y)$ is qdlt. $f$ induces a morphism between cone complexes $\Sigma(f):\Sigma(Y,D_Y)\to\Sigma(X,D_X)$. $\Sigma(f)$ is surjective, bijective on cones, and maps cones onto cones. If further $f$ is Galois, then $\Gal(f)$ acts on $\Sigma(Y,D_Y)$ and $\Sigma(Y,D_Y)/G\cong \Sigma(X,D_X)$ as real cone complexes.
    
    \item Let $f:Y \to X$ be a projective birational morphism. Let $D_Y=f^{-1}_*D_X+\Ex(f)$ and assume that $K_Y+D_Y=f^*(K_X+D_X)$. Then $(Y,D_Y)$ is tdlt and $\Sigma(Y,D_Y)$ is a subdivision of $\Sigma(X,D_X)$.
    
    \item Let $\Sigma'$ be a subdivision of $\Sigma(X,D_X)$. Then there is a tdlt modification  $f:(Y,D_Y)\to (X,D_X)$ such that $\Sigma(Y,D_Y)=\Sigma'$.

\end{enumerate}    
\end{prop}
\begin{proof}
(1) This follows from (3). 

(2) This follows from the fact that $U$ has rational singularities and \cite{AlexHacon2012}. 

(3) Let $S$ be a component of $D_X^{=1}$. We show that $S$ is normal. Let $\xi\in S\setminus U$ be a point with $\codim_S \xi \geq2$. Since $\xi$ is not an lc center, $\depth_{\xi}\Oo_X(-S) \geq3$ \cite[Theorem 7.20]{Kollar2013sing}. Since $X$ has rational singularities, $X$ is Cohen-Macaulay \cite[Lemma 5.12]{KM98}. Hence $\depth_{\xi}\Oo_X\geq2$. By $0\to \Oo_X(-S)\to \Oo_X \to \Oo_S\to0$, we have $\depth_{\xi}\Oo_S\geq2$ \cite[Lemma 2.60]{Kollar2013sing}. Hence $S$ is $S_2$. Let $\xi\in S\setminus U$ be a point with $\codim_S \xi =1$. Let $E$ be an exceptional divisor over $X$ such that $\xi \in C_X(E)$. Since $\codim_X \xi =2$, $\xi=C_X(E)$. Since $\xi$ is not an lc center, $a(E,X,D_X)>-1$. Hence $(X,D_X)$ is plt near $\xi$. Hence $D_X^{=1}$ is normal at $\xi$. In particular, $S$, being a component of $D_X^{=1}$, is $R_1$. Hence $S$ is normal by Serre's criterion.

Consider the pair $(S,D_S=\Diff^*_SD_X)$. By the toroidal adjunction, we have $D_S|_{S\cap U}=D_S^{=1}|_{S\cap U}$ and $(U,D_S|_{S\cap U})$ is toroidal. On the other hand, since lc centers of $(S,D_S)$ are lc centers of $(X,D_X)$, lc centers of $(S,D_S)$ intersects $S\cap U$. Hence $(S,D_S)$ is tdlt. Then (3) follows from an induction on dimension to $(S,D_S)$.

(4) By the Hurwitz formula, $K_Y+D_Y=f^*(K_X+D_X)$. By \cite[Proposition 5.20]{KM98}, $(Y,D_Y)$ is lc. Let $U_Y=f^{-1}U$. By $D_Y^{<1}=f^*D_X^{<1}$, we have $D_Y|_{U_Y}=D_Y^{=1}|_{U_Y}$. Let $\xi$ be an lc center of $(Y,D_Y)$. By the calculation in \cite[Proposition 5.20]{KM98}, $f(\xi)$ is an lc center of $(X,D_X)$. Hence $\xi$ intersects $U_Y$. By Lemma \ref{lem:toroidal} (1), $(U_Y,D_Y|_{U_Y})$ is toroidal without self-intersection. Hence $(Y,D_Y)$ is tdlt. 

Since $f|_{U_Y}$ is toroidal, there is a morphism $\Sigma(f|_{U_Y}): \Sigma(U_Y,D_Y|_{U_Y})\to \Sigma(U,D_X|_{U})$ between cone complexes \cite[Lemma 1.8]{WeakSemistableReduction2000}. Define $\Sigma(f):=\Sigma(f|_{U_Y})$. Let $T$ be an lc center, and let $S=f(T)$. From the local description (see the proof of Lemma \ref{lem:toric} (1) and Lemma \ref{lem:toroidal} (1)), we have that $f_*:N_T\to N_{S}$ is injective with finite index and $f_*:\sigma_T\to \sigma_S$ is bijective. Hence $\Sigma(f)$ is surjective, bijective on cones, and maps cones onto cones.

If $(X,D_X)$ is qdlt, then cones in $\Sigma(X,D_X)$ are simplcial. Hence cones in $\Sigma(Y,D_Y)$ are simplicial. Hence $(Y,D_Y)$ is qdlt.

Assume that $f$ is Galois. Then $\Sigma(f)\Sigma(g)(v)=\Sigma(f\circ g)(v)=\Sigma(f)(v)$ for $v\in\Sigma(X,D_X)$ and $g\in \Gal$. Let $v_1,v_2\in\Sigma(Y,D_Y)$ be valuations such that $\Sigma(f)(v_1)=\Sigma(f)(v_2)$. Then $T_1:=C_Y(v_1)$ and $T_2:=C_Y(v_2)$ are two lc centers of $(Y,D_Y)$ mapping to an lc center $S$ of $(X,D_X)$. Hence there is a $g\in \Gal$, such that $g(T_1)=T_2$. Hence $\Sigma(g)(v_1)\in\sigma_{T_2}$.
By $\Sigma(f)\Sigma(g)(v_1)=\Sigma(f)(v_1)=\Sigma(f)(v_2)$ and $\Sigma(f):\sigma_{T_2}\to\sigma_S$ is bijective, we have $\Sigma(g)(v_1)=v_2$. Hence $|\Sigma(Y,D_Y)|/G=|\Sigma(X,D_X)|$. Assume $T$ is an lc center and $g\in \Gal$ such that $g(T)=T$. Then by $\Sigma(f):\sigma_T\to\sigma_S$ is bijective, $\Sigma(g)|_{\sigma_S}$ is identity. Hence $\Gal(f)$ just permutes the members of the collection $\{(\sigma_T,N_{T})\}$.

(5) Let $U_Y=f^{-1}U$. By $D_Y^{<1}=f^{-1}_*D_X^{<1}$, we have $D_Y|_{U_Y}=D_Y^{=1}|_{U}$. Let $\xi$ be an lc center of $(Y,D_Y)$. Since $f(\xi)$ intersects $U$, $\xi$ intersects $U_Y$. By Lemma \ref{lem:toroidal} (2), $(U_Y,D_Y|_{U_Y})$ is toroidal without self-intersection. Hence $(Y,D_Y)$ is tdlt. Since $f|_{U_Y}$ is toroidal, $\Sigma(f):=\Sigma(f|_{U_Y})$ is a subdivision.

(6) The proof is identical to that of \cite[Lemma 36, Lemma 39]{de2017dual}.
\end{proof}

\begin{prop}\label{prop:tdlt-morphism}
Let $(X,D_X)$ be a projective tdlt Calabi-Yau pair, let $(Y,D_Y)$ be a toric compactification of a quasi-abelian variety, and let $f:X\to Y$ be a fibration. Assume that  $f^{-1}D_Y\subset D_X^{=1}$. Then the following hold.
\begin{enumerate}
\item The canonical bundle formula of $(X,D_X)\to Y$ is $K_X+D_X\sim_{\Qq}f^*(K_{Y}+D_Y)$. 
\item $f:(X,D_X)\to Y$ is locally stable over $Y\setminus D_Y$. In particular, $f$ is flat over $Y\setminus D_Y$.
\item $D_X^v$ is reduced and $D_X^v=f^{-1}D_Y$.
\item The $f$-image of any vertical lc center of $(X,D_X)$ is an lc center of $(Y,D_Y)$.
\item $f$ induces a morphism of cone complexes $\Sigma(f):\Sigma(X,D_X)\to\Sigma(Y,D_Y)$.
\end{enumerate}
\end{prop}
\begin{proof}

(1) Let $K_X+D_X\sim_{\Qq}f^*(K_{Y}+B+M)$ be the canonical bundle formula. By $f^{-1}D_Y\subset D_X^{=1}$, $B\geq D_Y$. Since $K_Y+D_Y\sim0$ and $M$ is pseudoeffective, we have $B=D_Y$ and $M\sim_{\Qq}0$.

(2) Let $\Sigma$ be a reduced snc divisor on $Y\setminus D_Y$. Then the canonical bundle formula for $f|_{f^{-1}(Y\setminus D_Y)}:(X,D_X+f^*\Sigma)|_{f^{-1}(Y\setminus D_Y)}\to Y\setminus D_Y$ is $K_X+D_X+f^*{\Sigma}|_{f^{-1}(Y\setminus D_Y)}\sim_{\Qq}f^*(K_{Y\setminus D_Y}+\Sigma)$. Since $(Y\setminus D_Y,\Sigma)$ is lc, by inversion of adjunction \cite[Theorem 3.1]{Ambro2004Shokurov} or \cite[Proposition 4.16]{filipazzi2018generalized}, $(X,D_X+f^*\Sigma)|_{f^{-1}(Y\setminus D_Y)}$ is lc. By \cite[Corollary 4.55]{Kollar-Families}, $f:(X,D_X)\to Y$ is locally stable over $Y\setminus D_Y$. 

(3) Let $Q$ be a component of $D_X^v$. Then $a(Q,X,D_X)<0$. By (2) and \cite[Corollary 4.56]{Kollar-Families}, $f(Q)\subset D_Y$. Hence $D_X^v\subset f^{-1}D_Y \subset D_X^{=1,v}$. Hence $D_X^v$ is reduced and $D_X^v=f^{-1}D_Y$.

(4) This follows from the inversion of adjunction.

(5) Let $S\subset X$ be an lc center. By $f^{-1}D_Y\subset D_X^{=1}$, $f$, via the pullback, induces linear maps $f^*:M_{f(S)}\to M_S$, $f^*:M^+_{f(S)}\to M^+_S$ and the dual $f_*:N_S\to N_{f(S)}$, $f_*:\sigma_S\to \sigma_{f(S)}$. The morphisms $f_*:\sigma_S\to \sigma_{f(S)}$, for all $S$, form a morphism between cone complexes $\Sigma(f):\Sigma(X,D_X)\to\Sigma(Y,D_Y)$.
\end{proof}

\begin{rem}
The definition of $\Sigma(f)$ is compatible with the slogan: Weil divisors (valuations) are meant to be
pushed-forward, whereas Cartier divisors (functions) are meant to be pulled-back \cite[Page 8]{boucksom2006differentiability}.
\end{rem}

\begin{rem}
In the setting of Proposition \ref{prop:tdlt-morphism}, one may say that $f$ is `toroidal at generic points of lc centers':
\begin{center}
\begin{tikzcd}
{\widehat{\Oo}_{X,\eta_S}} \arrow[r, "\cong"]                           & {k(\eta_S)[[M_S^{+}]]}                           \\
{\widehat{\Oo}_{Y,\eta_{f(S)}}} \arrow[r, "\cong"] \arrow[u, "\hat{f}^{\#}"] & {k(\eta_{f(S)})[[M_{f(S)}^{+}]]} \arrow[u, "\hat{g}^{\#}"]
\end{tikzcd}    
\end{center}
and the morphism $\Sigma(f):\Sigma(X,D_X)\to\Sigma(Y,D_Y)$ is the same as the construction for toroidal morphisms \cite[Lemma 1.8]{WeakSemistableReduction2000}.
\end{rem}

\begin{prop}\label{prop:vertical}
In the setting of Proposition \ref{prop:tdlt-morphism}, assume that $D_X^{=1}$ is vertical. Then the following holds. 
\begin{enumerate}
    \item $\Sigma(f)$ makes $\Sigma(X,D_X)$ a subdivision of $\Sigma(Y,D_Y)$ as real cone complexes.
    \item Assume that $\Sigma(f)$ maps cones onto cones. Then there is a Galois finite cover $q:(Z,D_Z)\to(Y,D_Y)$ which is \'etale over $Y^0:=Y\setminus D_Y$ such that $(X,D_X)\times_YZ\cong (F,D_F)\times (Z,D_Z)$ over $Z$. If $T$ is an lc center of $(Y,D_Y)$, then $f^{-1}T$ is irreducible and is an lc center of $(X,D_X)$.
\end{enumerate}    
\end{prop}
\begin{proof}
Let $X^0=f^{-1}Y^0$ and $D_{X^0}=D_X|_{X^0}$. Since $D_X^{=1}$ is vertical, $D_X^{=1}=f^{-1}D_Y$. Hence $(X^0,D_{X^0})$ is klt (Proposition \ref{prop:tdlt-basic} (1)). Consider the fibration $f|_{X^0}:(X^0,D_{X^0})\to Y^0$. By \cite[Proposition 4.4, Theorem 4.7]{Ambro2005moduli}, there is a Galois finite \'etale cover $Z^0\to Y^0$ such that there is an isomorphism \[(X^0,D_{X^0})\times_{Y^0}Z^0\cong (F,D_F)\times Z^0\] over $Z^0$, where $(F,D_F)$ is a fiber of $f|_{X^0}:(X^0,D_{X^0})\to Y^0$. Let $Z$ be the normalization of $Y$ in $K(Z^0)$, let $q:Z\to Y$ be the induced Galois finite cover, and let $D_Z=q^{-1}D_Y$. Let $X_Z$ be the normalization of $X\times_YZ$, let $p:X_Z\to X$ be the induced Galois finite cover, let $f_Z:X_Z\to Z$ be the induced fibration, and let $D_{X_Z}=p^{-1}D_X^{=1}+p^*D_X^{<1}$. Since $q$ is \'etale over $Y^0$, $p$ is \'etale over $X^0$. By Proposition \ref{prop:tdlt-basic}(4), $(X_Z,D_{X_Z})$ is tdlt and by Lemma \ref{lem:toric}(1), $(Z,D_Z)$ is a toric compactification of a quasi-abelian variety. Note that $(X_Z,D_{X_Z})$ and $(F,D_F)\times (Z,D_Z)$ are two compactifications of $(F,D_F)\times Z^0$. 

\begin{center}
\begin{tikzcd}
{(X,D_X)} \arrow[d, "f"] & {(X_Z,D_{X_Z})} \arrow[d, "f_Z"] \arrow[l, "p"'] & {(F,D_F)\times Z^0} \arrow[l, hook'] \arrow[r, hook] \arrow[d, "r^0"] & {(F,D_F)\times (Z,D_Z)} \arrow[d, "r"] \\
{(Y,D_Y)}                & {(Z,D_Z)} \arrow[l, "q"']                        & Z^0 \arrow[l, hook'] \arrow[r, hook]                                  & {(Z,D_Z)}                             
\end{tikzcd}   
\end{center}
Let $X_Z\xleftarrow{a}W\xrightarrow{b}F\times Z$ be a common log resolution and write \[a^*(K_{X_Z}+D_{X_{Z}})=b^*(K_{F\times Z}+D_F\times Z+F\times D_Z)+E_W.\] Then $E_W\sim_{\Qq}0$, $a_*E_W$ is supported on $D_{X_Z}^{=1}$ and $b_*E_W$ is supported on $F\times D_Z$. Let $P$ be a component of $D_{X_Z}^{=1}$. Then \[v_P(E_W)=-a(P,X_Z,D_{X_Z})+a(P,F\times Z,D_F\times Z+F\times D_Z)=1+a(P,F\times Z,D_F\times Z+F\times D_Z)\geq0\] Hence $a_*E_W\geq0$. By the negativity Lemma \cite[3.39]{KM98}, $E_W\geq0$. Applying the same argument to components of $F\times D_Z$, we deduce that $E_W\leq0$. Hence $(X_Z,D_{X_Z})$ and $(F,D_F)\times (Z,D_Z)$ are crepant birational over $Z$.

(1) Since $(F,D_F)$ is klt, under $\Sigma(r)$, we identity $\Sigma((F,D_F)\times(Z,D_Z))$ with $\Sigma(Z,D_Z)$. Since $(X_Z,D_{X_Z})$ and $(F,D_F)\times (Z,D_Z)$ are crepant birational, $|\Sigma(X_Z,D_{X_Z})|=|\Sigma((F,D_F)\times(Z,D_Z))|$ as subsets in $\Val_{X_Z}$. Hence $\Sigma(f_Z):\Sigma(X_Z,D_{X_Z})\to\Sigma(Z,D_Z)$ is bijective. By Proposition \ref{prop:tdlt-basic}(4), 
 $\Sigma(f)$ is bijective. Since $\Sigma(f)$ maps cones into cones, it is a subdivision of real cone complexes.

(2) Assume that $\Sigma(f)$ maps cones onto cones. Then $\Sigma(f_Z)$ maps cones onto cones. Hence the rational map $X_Z\dashrightarrow F\times Z$ contracts no divisors. Let $Z''\to Z$ be a birational morphism from a projective normal variety $Z''$. Then $\phi:X_Z\dashrightarrow F\times Z''$ contracts no divisors. Since $(X_Z,D_{X_Z})$ is tdlt, by Proposition \ref{prop:tdlt-basic}(2), $X_Z$ has rational singularities. By \cite[Proposition 18]{kollar2010quotients}, there are projective normal varieties $F'$, $Z'$, birational maps $a:F'\dashrightarrow F$, $b:Z'\dashrightarrow Z''$ and an isomorphism $X_Z\cong F'\times Z'$ such that $\phi=(a,b)$ as rational maps. Since $\phi$ is isomorphic along general fiber $f_Z^{-1}(z)$, $a$ is an isomorphic. Hence we may identify $F'$ with $F$. Let $\pi:X_Z \to Z'$ be the projection. Consider the birational map $s:Z'\dashrightarrow Z$. Since $\phi$ is over $Z$, the following diagram commutes.

\begin{center}
\begin{tikzcd}
X_Z\cong F\times Z' \arrow[d, "\pi"] \arrow[rd, "f_Z"] &   \\
Z' \arrow[r, "s", dashed]                               & Z
\end{tikzcd}    
\end{center}
Fix a point $x_0\in F$. For general $z'\in Z'$, $s(z')=f_Z(x_0,z')$, so $s$ coincides with the morphism $z'\mapsto f_Z(x_0,z')$. 

Let $D_{Z'}=s^{-1}(D_Z)$. Then the $f_Z$-vertical part of $D_{X_Z}$ is $f_Z^{-1}(D_Z)=\pi^{-1}D_{Z'}$. Since the $f_Z$-horizontal part of $D_{X_Z}$ is determined by its restriction to $Z^0$, $D_{X_Z}^h=D_F\times Z'$. Hence $D_{X_Z}=D_F\times Z'+D_{Z'}\times F$. Hence $K_{X_Z}+D_{X_Z}=\pi^*(K_{Z'}+D_{Z'})$. Hence $K_{Z'}+D_{Z'}=s^*(K_Z+D_Z)$. By Lemma \ref{lem:toric}(2), $(Z',D_{Z'})$ is a toric compactification of a quasi-abelian variety and $s$ is induced by a subdivision of $\Sigma(Z,D_Z)$. Since $\Sigma(f_Z)$ maps cones onto cones, $\Sigma(s)$ maps cones onto cones. Hence $s$ is in fact an isomorphism. Hence there is an isomorphism $(X_Z,D_{X_Z})\cong (F,D_F)\times (Z,D_Z)$ over $Z$.

Let $T$ be an lc center of $(Y,D_Y)$. By the proof of Lemma \ref{lem:toric}(1), we may identify $\Sigma(Z,D_Z)$ and $\Sigma(Y,D_Y)$ as real cone complexes, $q^{-1}T$ is irreducible and is an lc center of $(Z,D_Z)$. Since $X_Z\cong F\times Z$, $f^{-1}T=p(f_Z^{-1}(q^{-1}T))=p(q^{-1}T\times F)$ is irreducible and is an lc center of $(X,D_X)$. 
\end{proof}

\begin{proof}[Proof of Theorem \ref{thm:zhu3} for lc pairs]
Let $\pi:(Y,D_Y)\to (X,D_X)$ be a dlt modification with $\pi^{-1}\nklt(X,D_X)=D_Y^{=1}$ \cite[Lemma 29]{de2017dual}. Then (1) and (2) follow from the dlt case of Theorem \ref{thm:zhu3}. We focus on (3). 

\textit{Step 1. Let $A$ (resp. $T$) be the abelian (resp. torus) part of $G$. The composite morphism $Y\to A$ factors though $X$.}

By the dlt case of this theorem \ref{thm:zhu3}, there is a toric compactification $(Z,D_Z)$ of $G$, and a morphism $f:Y\to Z$ extending $\alpha\circ\pi|_{Y^0}$ such that via $f$, there is a one-to-one correspondence between lc centers of $(Y,D_Y)$ and $(Z,D_Z)$. Further, by the dlt case of \ref{thm:zhu3} (3-e), $f$ is equidimensional. Since $(Y,D_Y)$ is dlt, cones in $\Sigma(Z,D_Z)$ are simplicial.

\begin{center}
\begin{tikzcd}
 & Y^0 \arrow[r, hook] \arrow[d, "\alpha\circ\pi|_{Y^0}"] \arrow[ld, "\pi|_{Y^0}"'] & Y \arrow[d, "f"] \\
X^0 \arrow[r, "\alpha"] & G \arrow[r, hook]                                                                & Z               
\end{tikzcd}    
\end{center}

Let $D'\geq0$ be a nef divisor on $Z$ with $\Supp D'=D_Z$ (Lemma \ref{lem:nef}). Then $(Y,D_Y-\epsilon f^*D')$ is klt, $-(K_Y+D_Y-\epsilon f^*D')\sim_{\Qq}\epsilon f^*D'$ is nef over $X$ and $-K_Y$ is big over $X$. By \cite[Corollary 1.4]{hacon2007shokurov}, the fibers of $\pi$ are rationally chain connected. Hence $Y\to A$ factors through a morphism $X\to A$. 

Let $S\subset X$ be an lc center, and let $T$ be a minimal lc center of $(Y,D_Y)$ dominating $S$

\textit{Step 2. $\dim S=\dim T$.}
 
Let $\nu\circ\tilde{\pi}:T\to \tilde{S}\to S$ be the Stein factorization. Let $T^0:=T\setminus$ all lc centers of $(Y,D_Y)$ not containing $T$. Let $V^0=f(T^0)$. Then $V^0$ is a stratum of $(Z,D_Z)$, $f:T^0\to V^0$ is proper, and $\dim Y-\dim Z=\dim T-\dim V$. Note that $V^0$ is a quasi-abelian variety with abelian part $A$. Since $T$ is minimal among all lc centers dominating $S$, $\pi:T^0\to S^0:=\pi(T^0)$ is proper over the generic point of $S^0$. Hence $\tilde{\pi}:T^0\to \tilde{S}^0:=\nu^{-1}S^0$ is proper over the generic point of $\tilde{S}^0$. Consider the following commutative diagram.

 \begin{center}
\begin{tikzcd}
T^0 \arrow[d, "\tilde{\pi}"'] \arrow[r, "f"]           & V^0 \arrow[d, "q"] \\
\tilde{S}^0 \arrow[ru, "\beta",dashed] \arrow[r, "a"] & A                 
\end{tikzcd}
 \end{center}
where $q$ is the natural projection and $\tilde{S}^0\to A$ is the composite $\tilde{S}^0\to S^0\to X\to A$.

Let $s\in\tilde{S}^0$ be a general point. Since $\tilde{\pi}$ is proper over a neighborhood of $s$, $f$ is proper and the fiber $q^{-1}(a(s))$ is affine, by $f(\tilde{\pi}^{-1}(s))\subset q^{-1}(a(s))$, $f(\tilde{\pi}^{-1}(s))$ must be a point. By the rigidity Lemma \cite[5.3 Proposition]{KollarRationalCurves}, there is a rational map $\beta:\tilde{S}^0\dashrightarrow V^0$ such that $\beta\circ\tilde{\pi}=f$.

For general $z\in Z$, since $f$ extends $\alpha\circ \pi:Y^0\to G$, $\pi(Y_z)$ is a general fiber of $\alpha:X\to G$. Hence 
\begin{align}\label{equ:dim1}
\dim \pi(Y_z)=\dim X-\dim Z=\dim Y-\dim Z.    
\end{align} By the factorization $f=\beta\circ \tilde{\pi}$, after a possibly shrinking of $\tilde{S}^0$, for general $v\in V$, $\tilde{\pi}(Y_v)$ is a general fiber of $\beta:\tilde{S}^0\to V^0$. Hence 
\begin{align}\label{eqa:dim2}
\dim(\pi(Y_v))=\dim \tilde{\pi}(Y_v)=\dim S-\dim V.    
\end{align} 
Since the function $ Z\ni z\mapsto \dim\pi(Y_z)$ is upper semicontinuous, by \ref{equ:dim1} and \ref{eqa:dim2}, we have
\[\dim S-\dim V=\dim\pi(Y_v)\geq \dim\pi(Y_z)=\dim Y-\dim Z=\dim T-\dim V.\]
Hence $\dim S=\dim T$.

\textit{Step 3. $(X,D_X)$ is toroidal-dlt.}

Let $M_S$ (resp. $M_T$) be the lattice of Cartier divisors on $\Spec\Oo_{X,\eta_S}$ (resp. $\Spec\Oo_{Y,\eta_T}$) supported on $D_{X}^{=1}$ (resp. $D_Y^{=1}$). Let $M$ be the character lattice of the torus part of $G$. By step 1, $Y\to A$ factors through $X\to A$. Hence there is a factorization
\[M\to M_{S}\to M_T.\]

By the dlt case of this theorem,
$M_{\Rr}\to M_{T,\Rr}$ is surjective. Hence $M_{S,\Rr}\to M_{T,\Rr}$ is surjective. Hence by Step 2, we have \[\rank M_S\geq \rank M_T=\codim T=\codim S\] By Theorem \ref{thm:tdlt-criterion}, at general points of $S$, we have that $D_X=D_X^{=1}$ and $(X,D_X)$ is toroidal without self-intersection.
 
\textit{Step 4}. Since $(X,D_X)$ is toroidal-dlt, there is a corresponding cone complex $\Sigma(X,D_X)$. The remaining argument is the same as for dlt pairs.
\end{proof}

\begin{lem}\label{lem:nef}
Let $(G_{\Sigma},D_{\Sigma})$ be a projective toric compactification of a quasi-abelian variety $G$. Then there is an effective nef divisor $D'$ with $\Supp D'=D_{\Sigma}$.

\end{lem}
\begin{proof}
Let $D$ be a torus-invariant effective very ample divisor on the toric variety $X_{\Sigma}$ with $\Supp D$ equal to the toric boundary. Let $D'$ be the induced divisor on $G_{\Sigma}$. Then $\Supp D'=D_{\Sigma}$. We prove that $D'$ is nef. Let $q:G_{\Sigma}\to A$ be the projection. Since $D$ is very ample, $D'$ is very ample over $A$. Hence there is an embedding $\iota:G_{\Sigma}\hookrightarrow
\Pp_A(q_*\Oo(D'))$ with $D'=\iota^* \Oo(1)$. Since $H^0(X_{\Sigma},D)=\bigoplus_{m\in P_D\cap M}k\cdot \chi^m$ \cite[Proposition 4.3.3]{CoxLittleSchenckToric}, $q_*\Oo(D')=\bigoplus_{m\in P_D\cap M}L_m$. Since the $L_m$ are numerically trivial, $\Oo(1)$ is nef \cite[Lemma 2.3.2]{Positivity1}. Hence $D'=\iota^*\Oo(1)$ is nef.
\end{proof}

\section{Dual complexes}

\begin{defn}[{Qdlt pairs \cite[Proposition 34, Definition 35]{de2017dual}}]\label{def:qdlt}
Let $(X,D_X)$ be an lc pair. It is called quotient-dlt (qdlt for short) if for any lc center $S\subset X$, there are $\codim S$ components $D_i$ of $D_X^{=1}$, which are $\Qq$-Cartier near the generic point of $S$, such that $S\subset D_i$ for all $i$.
\end{defn}

Theorem \ref{thm:tdlt-criterion} and the following fact help to understand Definition \ref{def:qdlt}. Let $(X,D_X)$ be an lc pair. Then $(X,D_X)$ is dlt iff for any lc center $S\subset X$, there are $\codim S$ components $D_i$ of $D_X^{=1}$, which are Cartier near the generic point of $S$, such that $S\subset D_i$ for all $i$ \cite[18.23]{FlipsAbundanceThreefolds}.

\begin{lem}\cite[Lemma 2.5]{xu2025stable}\label{lem:qdltcenter}
Let $(X,D_X)$ be a qdlt pair. Then the codimension-$c$ lc centers of $(X,D_X)$ are normal, and are precisely the irreducible components of intersections of $c$ components of $D_X^{=1}$.
\end{lem}

\begin{defn}\cite[Definition 8]{de2017dual}
Let $\Dd$ be a $\Delta$-complex. $\Dd$ is called regular if any simplex has distinct vertices, and is called simplicial if it is regular and distinct simplices have distinct vertex sets. Let $\sigma\in\Dd$ be a simplex. Define the star $\Star(\sigma,\Dd):=\cup_{\sigma\leq\eta}\eta^0$ and the closed star $\overline{\Star}(\sigma,\Dd):=\cup_{\sigma\leq\eta}\eta$. If $\Dd$ is simplicial, define the link $\Link(\sigma,\Dd):=\{\eta\in \Dd\mid \eta\cap \sigma=\emptyset, \sigma\cup\eta\in \Dd\}$. We have the following standard facts: $\overline{\Star}(\sigma,\Dd)=\sigma*\Link(\sigma,\Dd)$ and $\overline{\Star}(\sigma,\Dd)\setminus \Star(\sigma,\Dd)=\partial\sigma *\Link(\sigma,\Dd)$ \cite[Lemma 62.6]{munkres1984elements}.

\textbf{Condition}\label{condition} Let $D_X=\sum_{i\in I}D_i\geq0$ be a reduced divisor on variety $X$, where the $D_i$ are prime. Assume that 

\begin{enumerate}
     \item the $D_i$ are normal and
     \item for every $J\subset I$, if $\cap_{i\in J}D_i\neq\emptyset$, then connected components of $\cap_{i\in J}D_i$ (called strata) are irreducible and have codimension $|J|$ in $X$.
\end{enumerate}
Then the dual complex $\Dd(D_X)$ is defined as a $\Delta$-complex, whose dimension-$c$ simplices correspond to codimension-$(c+1)$ strata, and the face relations correspond to reverse inclusion of strata. 

Let $(X,D_X)$ be a qdlt pair. By Lemma \ref{lem:qdltcenter}, $\Dd(D_X^{=1})$ is well-defined. Further, if $D_X^{=1}\geq\Delta_X\geq0$ is a reduced divisor, then $\Dd(\Delta)$ is also well-defined. If $(X,D_X)$ is just lc , then $\DMR(X,D_X)$ is defined as the dual complex of one of its dlt modifications. By Proposition \ref{prop:dFKX}(1), the PL-homeomorphism class of $\DMR(X,D_X)$ is independent of the choice of dlt modification.
\end{defn}

Let $(X,D_X)$ be a dlt pair. Let $S$ be an lc center and let $\sigma_S$ be the corresponding simplex in $\Dd(D_X^{=1})$. Let $D_S:=\Diff_S^*D_X$. If $\Dd(D_X^{=1}) $ is simplicial, then $\Link(\sigma_S,\Dd(D_X^{=1}))=\Dd(D_S^{=1})$. Indeed, if $D_X^{=1}=\sum_{i\in I}D_i$ and $S=\bigcap_{j\in J}D_j$, then the lc centers of $(S,D_S)$ are precisely the intersections $S\cap \bigcap_{k\in K}D_k$ where $K\subset I\setminus J$ and the intersection is nonempty.

\begin{prop}\cite[Proposition 11, Corollary 22, Proposition 40]{de2017dual}\label{prop:dFKX}
\begin{enumerate}
    \item Let $(X,D_X)$ and $(Y,D_Y)$ be qdlt pairs. If they are crepant birational, then their dual complexes are PL-homeomorphic.
    \item Let $(X,D_X)$ be a qdlt pair, let $f:X\to Z$ be a morphism, and let $p:X\dashrightarrow Y$ be a $(X,D_X)$-MMP over $Z$. Let $D_Y=p_*D_X$. Then $(Y,D_Y)$ is qdlt and $\Dd(D_Y)$ is a subcomplex of $\Dd(D_X)$. If there is a numerically $f$-trivial effective divisor whose support is $D_X^{=1}$, then $\Dd(D_X)$ collapses to $\Dd(D_Y)$.
    \item Let $(X,D_X)$ be a $\Qq$-factorial qdlt pair and let $f:X \to Y$ be a Mori fibration. If $D_X^{=1}$ is vertical, then $D_Y:=f(D_X^{=1})$ is a divisor, $(Y,D_Y)$ is $\Qq$-factorial qdlt, and there is a natural identification $\Dd(D_X)=\Dd(D_Y)$.
\end{enumerate}    
\end{prop}

\begin{rem}
In the following, $\Sigma(f):\Sigma(X,D_X)\to\Sigma(Y,D_Y)$ is defined using the pushforward of valuations as in Proposition \ref{prop:tdlt-morphism} (5). 

Let $(X,D_X)$ be a qdlt pair, and let $f:(Y,D_Y)\to (X,D_X)$ be a qdlt modification. Since $f$ is crepant, $\Sigma(f)$ is bijective. Hence $\Sigma(f)$ makes $\Sigma(Y,D_Y)$ a subdivision of $\Sigma(X,D_X)$. Assume that there is an effective $\Qq$-Cartier divisor $\Delta_X$ such that $D_X^{=1}=\Supp \Delta_X$ and $D_Y^{=1}=\Supp f^*\Delta_X$. Then $\Sigma(f)$ induces a PL-map \[\Dd(D_Y^{=1})=\{v\in\Sigma(Y,D_Y)\mid v(f^*\Delta_X)=1\}\to \Dd(D_X^{=1})=\{v\in\Sigma(X,D_X)\mid v(\Delta_X)=1\}.\] 
Hence $\Dd(D_X^{=1})$ and $\Dd(D_Y^{=1})$
not merely PL-homeomorphism as in Proposition \ref{prop:dFKX} (1), but $\Dd(D_Y^{=1})$ is a subdivision of $\Dd(D_X^{=1})$.

In the setting of Proposition \ref{prop:dFKX} (2).  Let $\sigma_T$ be a cone in $\Sigma(Y,D_Y)$, where $T$ is an lc center of $(Y,D_Y)$. Then $T$ is the pushforward of a non-contracted lc center $S$ of $(X,D_X)$. Hence, the MMP $p$ is isomorphic at the generic point of $S$, and we may identify $\sigma_S$ with $\sigma_T$. Under this identification, $\Sigma(Y,D_Y)$ is naturally a sub-cone complex of $\Sigma(X,D_X)$. This induces $\Dd(D_Y^{=1})$ a subcomplex of $\Dd(D_X^{=1})$.

In the setting of Proposition \ref{prop:dFKX} (3). By the proof of \cite[Proposition 40]{de2017dual}, which goes back to \cite[Chapter 5]{hogadi2009degenerations}, we have $\Sigma(f):\Sigma(X,D_X)\to \Sigma(Y,D_Y)$ is isomorphic as real cone complexes. This isomorphism induces the identification $\Dd(D_X^{=1})=\Dd(D_Y^{=1})$.
\end{rem}

\begin{lem}\label{lem:simplicial}
Let $(X,D)$ be a quasi-projective qdlt pair. Then there is a $\Qq$-factorial dlt modification $\pi:(Y,D_Y)\to(X,D_X)$ such that $\Dd(D_Y)$ is a subdivision of $\Dd(D_X)$ and $\Dd(D_Y)$ is simplicial.  
\end{lem}
\begin{proof}
Let $\pi':(Y',D_{Y'})\to (X,D_X)$ be a toroidal dlt modification as in \cite[Proposition 37]{de2017dual}. Let $X^0$ be the toroidal locus of $(X,D_X)$, let $Y'^0=\pi'^{-1}X^0$ and $D_{Y'}|_{Y'^0}=D_{Y'^0}$, then $(Y'^0,D_{Y'^0})$ is log smooth and $\Dd(D_{Y'})$ is a subdivision of $\Dd(D_X)$. Let $\pi'^0:(Y''^0,D_{Y''^0})\to(Y'^0,D_{Y'}|_{Y'^0})$ be the successive blowup of lc centers as in \cite[Remark 10]{de2017dual}. Then $\Dd(D_{Y''^0})$ is the barycentric subdivision of $\Dd(D_{Y'^0})=\Dd(D_{Y'})$. In particular, $\Dd(D_{Y''^0})$ is simplicial. By \cite[Lemma 36]{de2017dual}, there is a dlt modification $\pi'':(Y'',D_Y'')\to(Y',D_{Y'})$ extending $\mu$. Let $\pi''':(Y,D_Y)\to(Y'',D_{Y''})$ be a small $\Qq$-factorial modification \cite[Corollary 1.37]{Kollar2013sing}. Then $\Dd(D_Y)$ is identified with $\Dd(D_{Y''})$. Hence $\pi:=\pi'\circ\pi''\circ\pi''':(Y,D_Y)\to(X,D_X)$ is our desired.
\end{proof}

\begin{set}\label{setting}
Let $(X,D_X)$ be a projective dlt Calabi-Yau pair, let $(Y,D_Y)$ be a toric compactification of a quasi-abelian variety, and let $f:X\to Y$ be a fibration. Assume that $f^{-1}D_Y\subset D_X^{=1}$ and cones in $\Sigma(Y,D_Y)$ are simplicial. In particular, $Y$ is $\Qq$-factorial.
\end{set}

\begin{lem}\label{lem:slc}
In Setting \ref{setting}, the following hold.
\begin{enumerate}
\item There is a PL-map $\Dd(f):\Dd(D_X^v)\to\Dd(D_Y)$ induced by $f$.

\item Horizontal lc centers are precisely the strata of $(D_X^{h})^{=1}$.
   
\item The map $\Dd(f)$ is compatible with adjunction. More precisely, assume further that $\Dd(D_X)$ is simplicial. Let $S$ be a horizontal lc center such that $f|_{S}$ has connected fibers. Let $D_S:=\Diff^*_SD_X$. Then $\Dd(D^v_S)$ is sub-complex of $\Dd(D_X^v)$ and there is a commutative diagram 
\begin{center}
\begin{tikzcd}
\Dd(D_S^v) \arrow[r, hook] \arrow[d, "\Dd(f|_S)"] & \Dd(D_X^v) \arrow[d, "\Dd(f)"] \\
\Dd(D_Y) \arrow[r,  Rightarrow, no head]   & \Dd(D_Y).                      
\end{tikzcd}
\end{center}

\end{enumerate}   
\end{lem}
\begin{proof}

(1) By Proposition \ref{prop:tdlt-morphism}, we have that $f:(X,D_X)\to Y$ is locally stable over $Y\setminus D_Y$, $D_X^v$ is reduced, $D_X^v=\Supp f^*D_Y$ and there is a morphism of cone complexes $\Sigma(f):\Sigma(X,D_X)\to\Sigma(Y,D_Y)$. Let $\Sigma(X,D_X^v)$ be the sub-cone complex of $\Sigma(X,D_X)$ corresponding to $\Dd(D_X^v)$. Then $\Sigma(f)$ induces a PL-map
\[\Dd(D_X^v)=\{v\in \Sigma(X,D_X^v)\mid v(f^*D_Y)=1\}\to \Dd(D_Y)=\{v\in \Sigma(Y,D_Y)\mid v(D_Y)=1\}.\] 

(2) Let $S\subset X$ be a component of an intersection of components of $(D_X^h)^{=1}$. Then $S$ is contained in $\codim S$ components of $(D_X^h)^{=1}$. Hence $S$ is not contained in any component of $D_X^v$. Hence $S$ is horizontal by Proposition \ref{prop:tdlt-morphism} (4). The converse is clear.

(3) 
It is clear that $K_S+D_S\sim_{\Qq}(f|_{S})^*(K_Y+D_Y)$. Since $f^{-1}D_Y\subset D_X^{=1}$, $(f|_{S})^{-1}D_Y\subset D_S^{=1}$. Hence $f|_{S}:(S,D_S)\to Y$ satisfies the condition of the Setting, and the components of $D_S^v$ are precisely the intersections of $S$ with components of $D_X^v$. Hence for any stratum $T$ of $(S,D_S^v)$, $T=S\cap T^v$ for an unique stratum $T^v$ of $(X,D_X^v)$. Write $D_X^v=\sum_{i\in I} D_i$. The map $\sigma_T\to\sigma_{T^v}$, $v=\sum_i\lambda_iv_{D_i|_ S}\mapsto \tilde{v}=\sum_i\lambda_iv_{D_i}$ identify $\sigma_T$ with $\sigma_{T^v}$ and makes $\Sigma(S,D_S^v)$ a sub-cone complex of $\Sigma(X,D_X^v)$. Let $L$ be a $\Qq$-Cartier divisor supported on $D_X^v$. Then $v(L|_{S})= \tilde{v}(L)$. Hence, the following diagram commutes.

\begin{center}
\begin{tikzcd}
v \arrow[d, maps to] \arrow[rrr, maps to, bend left] & {\Sigma(S,D_S^v)} \arrow[r, hook] \arrow[d, "\Sigma(f|_{S})"] & {\Sigma(X,D_X^v)} \arrow[d, "\Sigma(f)"] & \tilde{v} \arrow[d, maps to]        \\
(\Delta\mapsto v((f|_{S})^*\Delta)                   & {\Sigma(Y,D_Y)} \arrow[r, Rightarrow, no head]                & {\Sigma(Y,D_Y)}                          & (\Delta\mapsto\tilde{v}(f^*\Delta))
\end{tikzcd}    
\end{center}
Then $\Dd(D_S^v)=\{v\in\Sigma(S,D_S^v)\mid v((f|_{S})^*D_Y)=1\}$ is naturally a subcomplex of $\Dd(D_X^v)=\{v\in\Sigma(X,D_S^v)\mid v(f^*D_Y)=1\}$ and the diagram in (3) commutes.
\end{proof}

\begin{prop}\label{prop:inductivezero}
In Setting \ref{setting}, the map $\Dd(f):\Dd(D_X^v)\to\Dd(D_Y)$ is a homotopy equivalence.  
\end{prop}
\begin{proof}

\textit{Step 1}. If $D_X^v=D_X$, then by Proposition \ref{prop:vertical}, $\Sigma(f):\Sigma(X,D_X^v)\to\Sigma(Y,D_Y)$ is a subdivision. In particular, $\Dd(f):\Dd(D_X^v)\to\Dd(D_Y)$ is a homotopy equivalence. Hence we may assume that $D_X^v\neq D_X$. Let $\pi:(X',D_{X'})\to (X,D_X)$ be a small $\Qq$-factorial modification \cite[Corollary 1.37]{Kollar2013sing}. Since $\pi$ is isomorphic over the log smooth locus of $(X,D_X)$. We may identify $\Sigma(X',D^v_{X'})$ with $\Sigma(X,D^v_X)$. Hence we may assume from the beginning that $X$ is $\Qq$-factorial. Then $K_X+D_X^v\sim_{\Qq}-D_X^h$ is not pseudoeffective over $Y$. 
 
\textit{Step 2}. By running a $(X,D_X^v)$-MMP over $Y$, we have a Mori fibration $\pi:Z\to W$ over $Y$. Let $g:W\to Y$ be the induced fibration.

\begin{center}
\begin{tikzcd}
X \arrow[dd, "f"] \arrow[r, "\text{MMP}", dotted] & Z \arrow[d, "\pi"] \\
 & W \arrow[ld, "g"]  \\
Y   &    
\end{tikzcd}    
\end{center}

Let $D_Z$ be the pushforward of $D_X$ on $Z$. Then $(Z,D_Z)$ is $\Qq$-factorial lc Calabi-Yau, $D_Z^v=(g\circ \pi)^{-1}D_Y$ is the pushforward of $D_X^v$, and $(Z,D_Z^v)$ is dlt. Let $K_Z+D_Z\sim_{\Qq}\pi^*(K_W+B_W+M_W)$ be the canonical bundle formula of $\pi:(Z,D_Z)\to W$. By \cite[Theorem 1.5]{bakker2025baily}, there is an effective $\Qq$-divisor $B_W'\sim_{\Qq}M_W$ such that $(W,D_W:=B_W+B_W')$ is lc and $D_W^{=1}=B_W^{=1}$. By the construction of the discriminant divisor $B_W$, we have $g^{-1}D_Y=\pi(D_Z^v)\subset B_W^{=1}=D_W^{=1}$. Hence by (the same argument of) Proposition \ref{prop:tdlt-morphism} (3), $D_W^v=g^{-1}D_Y=\pi(D_Z^v)$. By Proposition \ref{prop:dFKX}(3), $W$ is $\Qq$-factorial, $(W,D_W^v)$ is qdlt and $\Dd(D_W^v)$ is naturally identified to $\Dd(D_Z^v)$ under $\Dd(\pi)$. 
 
Since $D_X^v=f^{-1}D_Y$ fully supports the numerically $f$-trivial effective divisor $f^*D_Y$, by Proposition \ref{prop:dFKX}(2), $\Dd(D_X^v)$ collapses to $\Dd(D_Z^v)$. Hence the inclusion $\Dd(D_Z^v)\to\Dd(D_X^v)$ is a homotopy equivalence. From our definition of $\Dd(f)$, $\Dd(g\circ \pi)$, and $\Dd(g)$, the following diagram commutes.

\begin{center}
\begin{tikzcd}
\Dd(D_X^v) \arrow[dd, "\Dd(f)"] & \Dd(D_Z^v) \arrow[l, hook'] \arrow[d, Rightarrow,no head] \\
& \Dd(D_W^v) \arrow[ld, "\Dd(g)"]\\
\Dd(D_Y)                        &           
\end{tikzcd}    
\end{center}
Hence, to show that $\Dd(f):\Dd(D_X^v)\to\Dd(D_Y)$ is a homotopy equivalence, it suffices to show that $\Dd(g):\Dd(D_W^v)\to \Dd(D_Y)$ is a homotopy equivalence.

\textit{Step 3}. 
We want to apply induction on dimension to $(W,D_W)$. The problem is that $(W,D_W)$ may not be dlt. 

Let $\pi:(W',D_{W'})\to (W,D_W)$ be a $\Qq$-factorial dlt modification and let $g'=g\circ \pi$. Then $g':(W',D_{W'})\to Y$ satisfies the Setting \ref{setting}. By induction on dimension, $\Dd(g'):\Dd(D_{W'}^v)\to\Dd(D_Y)$ is a homotopy equivalence. We show that $\Dd(g)$ is also a homotopy equivalence. Note that since $\pi:(W',D_{W'}^v)\to (W,D_W^v)$ may not be crepant, there may not exist a map $\Dd(\pi):\Dd(D_{W'}^v)\to\Dd(D_W^v)$. Hence we can not say that since $\Dd(\pi)$ is a homotopy equivalence, then $\Dd(g)$ is a homotopy equivalence.

Let $\Ex'(\pi)=\sum_{E\subset \Ex(\pi)\setminus D_{W'}^v}E$ be the sum of components of $\Ex(\pi)$ which are not contained in $D_{W'}^v=\pi^{-1}D_W^v$. Since $(W,D_W^v)$ is qdlt, its non-klt locus is $D_W^v$. Hence if $E$ is a component of $\Ex'(\pi)$, then $a(E,W,D_W^v)>-1$. Let $0<\epsilon\ll1$, then $E_Y$ in the following formula is effective, exceptional and $\Ex'(\pi)\subset \Supp E_Y$. \[K_{W'}+D_{W'}^v+(1-\epsilon)\Ex'(\pi)=\pi^*(K_W+D_W^v)+E_Y\]
Let $(W'',D^v_{W''})$ be a good minimal model of $(W',D_{W'}^v+(1-\epsilon)\Ex'(\pi))$ over $W$ and let $\pi':W''\to W$ be the induced birational morphism. Then $\pi':(W'',D^v_{W''})\to (W,D_W^v)$ is crepant, $D^v_{W''}$ is reduced, and $D^v_{W''}=(g\circ \pi')^{-1}D_Y$. Since $D_{W'}^v$ fully supports the $\pi$-numerically trivial effective divisor $(g\circ\pi)^*D_Y$, by Proposition \ref{prop:dFKX}(2), $\Dd(D_{W'}^v)$ collapses to $\Dd(D_{W''}^v)$. Since $\pi':(W'',D_{W''}^v)\to (W,D_W^v)$ is crepant, $g^{-1}D_Y=D_W^v$, and $(g\circ\pi')^{-1}D_Y=D_{W''}^v$, $\Dd(\pi'):\Dd(D_{W''}^v)\to\Dd(D_W^v)$ is a subdivision. Since the following diagram commutes, we conclude that $\Dd(g):\Dd(D_W^v)\to\Dd(D_Y)$ is a homotopy equivalence.

\begin{center}
\begin{tikzcd}
\Dd(D_{W'}^v) \arrow[dd, "\Dd(g')"] & \Dd(D_{W''}^v) \arrow[l, hook'] \arrow[d, "\Dd(\pi')"] \\
  & \Dd(D_W^v) \arrow[ld, "\Dd(g)"]\\
\Dd(D_Y) &                                              
\end{tikzcd}    
\end{center} 
\end{proof}

\begin{prop}\label{prop:connected}
In Setting \ref{setting}. 
Assume further that $f|_{S}:S\to Y$ has connected fibers for any horizontal lc center $S$. Then there is a homotopy equivalence $\Dd(D_X^{=1})\cong \Dd(D_X^{h,=1})*\Dd(D_Y)$.
\end{prop}
\begin{proof}

The connected-fiber assumption essentially says that the spring of $f:(X,D_X)\to Y$ is $Y$ \cite[Theorem-Definition 4.45]{Kollar2013sing}. By this assumption and Lemma \ref{lem:slc} (2), we may identify $\Dd(D_X^{h,=1})$ with $\Dd(D_{X_{\overline{K(Y)}}}^{=1})$, where $X_{\overline{K(Y)}}$ is the geometric generic fiber of $f$. Let $\pi:(Z,D_Z)\to (X,D_X)$ be a dlt modification as in Lemma \ref{lem:simplicial}. Then $\Dd(D_Z)$ is simplicial. By \cite[Theorem-Definition 4.45 (7)]{Kollar2013sing}, the spring of $f\circ\pi:(Z,D_Z)\to Y$ is also $Y$. Since $(Z_{\overline{K(Y)}},D_{Z_{\overline{K(Y)}}})$ and $(X_{\overline{K(Y)}},D_{X_{\overline{K(Y)}}})$ are crepant birational, they are PL-homeomorphism by Proposition \ref{prop:dFKX} (1). Hence we may assume from the beginning that $\Dd(D_X^{=1})$ is simplicial.

For a simplex $\sigma$ in $\Dd(D_X^{=1})$, we denote by $S(\sigma)$ the corresponding stratum. Let $S\subset X$ be an stratum of $(X,D_X)$. If $S$ is not a stratum of $(X,D_X^v)$, then there is an unique stratum $S^h$ of $(X,D_X^h)$ such that either $S^h=S$ or $S$ is a stratum of $(S^h,D_{S^h}^v)$. We call $S^h$ the horizontal part of $S$. Note that if $S\subset T$, then $S^h\subset T^h$. We also use $\sigma^h$ to denote the horizontal part of a simplex $\sigma$ in $\Dd(D_X^{=1})$.

Order the simplices of $\Dd(D_X^{h,=1})$, $\sigma_1,\sigma_2,\cdots,\sigma_N$, such that all proper faces of $\sigma_i$ appears before $\sigma_i$. Let $A_0=D(D_X^v)$. Let $A_i$, $1\leq i\leq N$, be the subcomplex of $\Dd(D_X^{=1})$ whose simplices correspond to lc centers whose horizontal part is either empty or one of $\{ S(\sigma_1),S(\sigma_2),\cdots,S(\sigma_i)\}$. Then $A_N=\Dd(D_X^{=1})$. Let $B_0=\Dd(D_Y)$ and let $B_{i}=B_{i-1}\cup_{\partial\sigma_i*\Dd(D_Y)}\sigma_i*\Dd(D_Y)$, $0\leq i\leq N$. Then  $B_N=\Dd(D_X^{h,=1})*\Dd(D_Y)$. We claim the following.

\begin{enumerate}
\item $A_i\setminus\Star(\sigma_i,A_i)=A_{i-1}$.

\item $\Link(\sigma_i,A_i)=\Dd(D^v_{S(\sigma_i)})$,  

\item $A_i= A_{i-1}\cup_{\partial\sigma_i*\Link(\sigma_i,A_i)}\sigma_i*\Link(\sigma_i,A_i)$, and

\item for all $0\leq i\leq N$, there is a homotopy equivalence $\phi_i:A_{i}\to B_{i}$ such that if $j\leq i$, then the following diagram is commutative. 
\begin{center}\label{diag:commu}
\begin{tikzcd}
A_i \arrow[d, "\phi_i"] & \sigma_j*\Dd(D^v_{S(\sigma_j)}) \arrow[l, hook'] \arrow[d, "\id_{\sigma_j}*\Dd(f_j)"] \\
B_i                     & \sigma_j*\Dd(D_Y) \arrow[l, hook'] 
\end{tikzcd}    
\end{center}
\end{enumerate}
In particular, if (4) holds, then there is a homotopy equivalence $\Dd(D_X^{=1})\cong\Dd(D_X^{h,=1})*\Dd(D_Y)$. Now we prove the claim.

(1) This follows from $A_i\setminus \Star(\sigma_i,A_i)=\{\eta\in A_{i}\mid \sigma_i\not\leq\eta\}=A_{i-1}$.

(2) Strata corresponding to simplices in $\Link(S(\sigma_i),A_i)$ are nonempty intersections $\bigcap_{D_j\subset D_X^v} D_j$ with $S(\sigma_i)\cap \bigcap_{D_j\subset D_X^v} D_j\neq\emptyset$. Hence (2) follows.

(3) This follows from (1), (2) and the standard simplicial complex facts: $\overline{\Star}(\sigma_i,A_i)\setminus \Star(\sigma_i,A_i)=\partial \sigma_i*\Link(\sigma_i,A_i)$ and $A_i= A_i\setminus \Star(\sigma_i,A_i)\cup_{\overline{\Star}(\sigma_i,A_i)\setminus \Star(\sigma_i,A_i)}\overline{\Star}(\sigma_i,A_i)$.

(4) We construct the homotopy equivalence $\phi_i$ inductively. Let $f_i:=f
|_{S(\sigma_i)}$, $1\leq i\leq N$. Let $\phi_0=\Dd(f):\Dd(D_X^v)\to\Dd(D_Y)$. By Proposition \ref{prop:inductivezero}, $\phi_0$ is a homotopy equivalence. Assume that we have already constructed the homotopy equivalence $\phi_{i-1}$. Let $\sigma_j\leq\sigma_i$ be a proper face. Then $S(\sigma_i)$ is a horizontal lc center of $f_j:(S(\sigma_j),D_{S(\sigma_j)})\to Y$. Consider the following diagram.

\begin{center}
\begin{tikzcd}
A_{i-1} \arrow[d, "\phi_{i-1}"] & \sigma_j*\Dd(D^v_{S(\sigma_j)}) \arrow[l, hook'] \arrow[d, "\id_{\sigma_j}*\Dd(f_j)"] & \sigma_j*\Dd(D^v_{S(\sigma_i)}) \arrow[d, "\id_{\sigma_j}*\Dd(f_i)"] \arrow[l, hook'] \\
B_{i-1}                         & \sigma_j*\Dd(D_Y) \arrow[r, Rightarrow,no head] \arrow[l, hook']                     & \sigma_j*\Dd(D_Y)           
\end{tikzcd}    
\end{center}
By the inductive hypothesis, the first square commutes. By the compatibility (Lemma \ref{lem:slc}(3)), the second square commutes. Hence the first square in the following diagram commutes.
\begin{center}
\begin{tikzcd}
A_{i-1} \arrow[d, "\phi_{i-1}"] & \partial\sigma_i*\Dd(D_{S(\sigma_i)}^v) \arrow[d, "\id_{\partial\sigma_i}*\Dd(f_i)"] \arrow[r, hook] \arrow[l, hook'] & \sigma_i*\Dd(D_{S(\sigma_i)}^v) \arrow[d, "\id_{\sigma_i}*\Dd(f_i)"] \\
B_{i-1}                         & \partial\sigma_i*\Dd(D_Y) \arrow[r,hook] \arrow[l, hook']                                    & \sigma_i*\Dd(D_Y)                                 
\end{tikzcd}    
\end{center}
Since the second square is automatically commutative, we can define $\phi_i$ to be the pushout \[\phi_{i-1}\cup_{\id_{\partial\sigma_i}*\Dd(f_i)}\id_{\sigma_i}*\Dd(f_i):A_{i-1}\cup_{\partial\sigma_i*\Dd(D^v_{S(\sigma_i)})}\sigma_i*\Dd(D^v_{S(\sigma_i)}) \to B_{i-1}\cup_{\partial\sigma_i*\Dd(D_Y)}\sigma_i*\Dd(D_Y).\]
Since $\phi_{i-1}$ is a homotopy equivalence (by the inductive hypothesis) and $\Dd(f_i)$ is a homotopy equivalence,  $\phi_i:A_i\to B_i$ is a homotopy equivalence. By the definition of pushout, the commutativity of the diagram in claim (4) is clear.
 \end{proof}

\begin{prop}\label{prop:basechange}
In Setting \ref{setting}. Let $S\subset X$ be a minimal horizontal lc center and let $q\circ f_S:S\to Z \to Y$ be the Stein factorization of $f|_{S}$. Then $q$ is Galois and is \'etale over $Y\setminus D_Y$. Let $X_Z$ be the normalization of $X\times_{Y}Z$, let $p:X_Z\to X$ be the induced finite cover, let $f_Z:X_Z\to Z$ be the induced fibration, and let $D_{X_Z}$ be the divisor on $X_Z$ such that $K_{X_Z}+D_{X_Z}=p^*(K_X+D_X)$ holds. Then $(X_Z,D_{X_Z})$ is qdlt, $f_Z^{-1}D_Z\subset D_{X_Z}^{=1}$, and $f_Z|_T$ has connected fibers for any horizontal lc centers $T$ of $(X_Z,D_{X_Z})$.
\end{prop}

\begin{proof}

By \cite[Theorem–Definition 4.45 (5)]{Kollar2013sing}, $q$ is Galois. We show that $q$ is \'etale over $Y\setminus D_Y$. Let $P\subset Y$ be a prime divisor with $P\not\subset D_Y$. Then $(X,D_X+f^*P)$ is lc over the generic point of $P$. Let $D_S=\Diff^*_SD_X$. By adjunction, $(S,D_S+f^*P|_{S})$ is lc over the generic point of $P$. In particular, $f^*P|_S=(f|_S)^*P$ is reduced. Hence $q:Z\to Y$ is \'etale over the generic point of $P$. By the purity of the branch locus, $q$ is \'etale over $Y\setminus D_Y$. 

Since $q$ is \'etale over $Y\setminus D_Y$, $p$ is \'etale over $X\setminus D_X^{=1}$. Hence by Proposition \ref{prop:tdlt-basic} (4), $(X_Z,D_{X_Z})$ is qdlt and by Lemma \ref{lem:toric} (1), $(Z,D_Z)$ is a toric compactification of a quasi-abelian variety.

By construction, $f_Z|_{S'}$ has connected fibers for a minimal horizontal lc center $S'$ of $(X_Z,D_{X_Z})$. Hence, by \cite[Theorem–Definition 4.45 (2)]{Kollar2013sing}, $f_Z|_{T}$ has connected fibers for any horizontal lc center $T$.
\end{proof}

\begin{rem}[On coregularity and complexity]\label{rem}
Let $(X,D_X)$ be a projective lc Calabi-Yau pair, let $X^0$ be its klt locus, and let $\alpha:X^0\to G$ be the quasi-Albanese morphism. Let $A$ be the abelian part of $G$. We define the coregularity \[\coreg(X,D_X):=\dim X-\dim\DMR(X,D_X)-1,\] and the complexity \[c(X,D_X)=\dim X-|D_X^{=1}|+\rank\la D_X^{=1}\ra,\]
 where $|D_X^{=1}|$ is the number of components of $D_X^{=1}$, and $\rank\la D_X^{=1}\ra$ is the rank of the subgroup generated by the components of $D_X^{=1}$ in the group of Weil divisors modulo algebraic equivalence. Note that we have adopted a definition slightly different from that in \cite{Brown2018GeomToric}. If $(X,D_X)$ is dlt, then $\coreg(X,D_X)$ is just the dimension of its minimal lc centers. By \cite[Lemma 1.5]{zhu2025quasi}, $c(X,D_X)=\dim X-\dim G+\dim A$. 
 
We claim that \[ c(X,D_X)\geq\coreg(X,D_X),\] and the equality holds iff $\alpha$ is proper. Indeed, let $g:(Z,D_Z)\to\Pp_A$ be the fibration given by Theorem \ref{thm:zhu3}. If $D_Z^{=1}$ is vertical, then by Proposition \ref{prop:vertical}, $c(Z,D_Z)=\coreg(Z,D_Z)$. If $D_Z^{=1}$ is not vertical, then an induction on dimension to a horizontal lc center $S$ shows that \[\coreg(Z,D_Z)=\coreg(S,D_S)\leq\dim S-\dim G+\dim A<\dim X-\dim G+\dim A=c(X,D_X).\]  Since $\coreg(X,D_X)=\coreg(Z,D_Z)$, $c(X,D_X)=c(Z,D_Z)$ (\cite[Lemma 3.8]{zhu2025quasi}) and $\alpha$ is proper iff $D_Z^{=1}$ is vertical, the claim follows. See \cite[Theorem 1.4]{mauri2025birational} for a similar result. Hence, Theorem \ref{thm1} describes Calabi-Yau pairs with $\coreg(X,D_X)=c(X,D_X)$.

There is another inequality. Let $X\dashrightarrow Z$ be the MRC fibration \cite[IV.5.2]{KollarRationalCurves}. Then any lc centers of $(X,D_X)$ dominate $Z$ \cite[Proposition 19]{KollarXu2016dual}. Hence, \[c(X,D_X)\geq\coreg(X,D_X)\geq \dim Z.\] See \cite[Theorem 1.6]{Brown2018GeomToric} for a similar result.
\end{rem}

\section{Proof of the main theorems}

\begin{proof}[Proof of theorem \ref{thm1}]
This follows from Theorem \ref{thm:zhu1} and Theorem \ref{thm:zhu3}.  
\end{proof}

\begin{proof}[Proof of Corollary \ref{cor}]
By \cite[Theorem 1.2 and Theorem 1.3 (2)]{zhu2025quasi}, $\alpha:X^0\to G$ is flat with irreducible general fibers. Since $\dim X^0=\dim G$, $\alpha$ is quasi-finite birational. Since $G$ is normal, Zariski's main theorem implies that $\alpha$ is an open embedding. By \cite[Theorem 1.3 (1)]{zhu2025quasi}, $\codim G\setminus\alpha(X^0)\geq 2$. Hence $\alpha$ is proper in codimension one. Hence, by Theorem \ref{thm1}, $\alpha$ is in fact an isomorphism and $(X,D_X)$ is a toric compactification of $G$.
\end{proof}

\begin{rem}
In the setting of Corollary \ref{cor}, when $D_X$ is reduced, to show that $\alpha$ is proper in codimension one, one can also use \cite{mendes2023footnote}. 
\end{rem}

\begin{proof}[Proof of Theorem \ref{thm2}]
This follows from Theorem \ref{thm:zhu1}, Theorem \ref{thm:zhu3}, Proposition \ref{prop:connected} and Proposition \ref{prop:basechange}.        
\end{proof}



\bibliographystyle{alpha}

\bibliography{bibfile}

\end{document}